\documentclass[12pt]{amsart}

\usepackage{amsfonts,amsthm,amsmath,amssymb,amscd,mathrsfs}
\usepackage{graphics}
\usepackage{indentfirst}
\usepackage{cite}
\usepackage{latexsym}
\usepackage[dvips]{epsfig}

\usepackage{color}

\newtheorem{theorem}{Theorem}[section]
\newtheorem{remark}{Remark}[section]
\newtheorem{definition}{Definition}[section]
\newtheorem{lemma}[theorem]{Lemma}
\newtheorem{pro}[theorem]{Proposition}

\renewcommand{\div}{{\rm div }}

\newcommand{\bt}{\begin{theorem}}
\newcommand{\bl}{\begin{lemma}}
\newcommand{\el}{\end{lemma}}
\newcommand{\et}{\end{theorem}}

\newcommand{\bn}{\begin{eqnarray}}
\newcommand{\en}{\end{eqnarray}}
\newcommand{\bnn}{\begin{eqnarray*}}
\newcommand{\enn}{\end{eqnarray*}}
\newcommand{\ba}{\begin{aligned}}
\newcommand{\ea}{\end{aligned}}
\newcommand{\be}{\begin{equation}}
\newcommand{\ee}{\end{equation}}

\newcommand{\Bv}{{\boldsymbol{v}}}

\newcommand{\Bu}{{\boldsymbol{u}}}
\newcommand{\Be}{{\boldsymbol{e}}}
\newcommand{\BF}{{\boldsymbol{F}}}

\newcommand{\Bz}{{\boldsymbol{\zeta}}}

\newcommand{\BU}{\boldsymbol{U}}

\newcommand{\Bo}{{\boldsymbol{\omega}}}

\newcommand{\mcA}{\mathcal{A}}
\newcommand{\mcH}{\mathcal{H}}

\newcommand{\mcL}{\mathcal{L}}
\newcommand{\mcG}{\mathcal{G}}
\newcommand{\inte}{\int_{-1}^1}

\newcommand{\Z}{\mathbb{Z}}
\newcommand{\T}{\mathcal{T}}
\newcommand{\B}{\mathcal{B}}
\newcommand{\Q}{\mathcal{Q}}
\begin{document}

\title[Stability and Uniqueness of Poiseuille Flows]
{Uniform structural stability and uniqueness of Poiseuille flows in a two dimensional periodic strip}

\author{Kaijian Sha}
\address{Department of mathematics, East China University of Science and Technology, Shanghai, China}
\email{10132229@mail.ecust.edu.cn}

\author{Yun Wang}
\address{School of Mathematical Sciences, Center for dynamical systems and differential equations, Soochow University, Suzhou, China}
\email{ywang3@suda.edu.cn}

\author{Chunjing Xie}
\address{School of mathematical Sciences, Institute of Natural Sciences,
Ministry of Education Key Laboratory of Scientific and Engineering Computing,
and SHL-MAC, Shanghai Jiao Tong University, 800 Dongchuan Road, Shanghai, China}
\email{cjxie@sjtu.edu.cn}

\begin{abstract}
In this paper, we prove the uniform nonlinear structural stability of Poiseuille flows with arbitrarily large flux for the Navier-Stokes system in a two dimensional periodic strip when the period is not large. The key point is to establish the a priori estimate for the associated linearized problem via the careful analysis for the associated boundary layers. Furthermore, the well-posedness theory for the Navier-Stokes system is also proved even when the external force is large in $L^2$. Finally, if the vertical velocity is suitably small where the smallness is independent of the flux, then Poiseuille flow is the unique solution of the steady Navier-Stokes system in the periodic strip.
\end{abstract}


\thanks{Updated on \today}

\maketitle

\section{Introduction and Main Results}
The famous Leray problem (\cite{Galdi}) is to study the well-posedness for the steady Navier-Stokes system
\begin{equation}\label{NS}
\left\{
\begin{aligned}
&\boldsymbol{u}\cdot \nabla \boldsymbol{u} +\nabla p=\Delta \boldsymbol{u}+\BF,\\
&\div~\boldsymbol{u}=0,
\end{aligned}
\right.
\end{equation}
 in an infinitely long channel with no slip boundary conditions such that the solutions tend to the shear flows at far fields, where $\Bu=(u_1,u_2)$ and $\BF=(F_1,F_2)$ are respective the velocity field and external force in the two dimensional case.  When the far field of the channel tends to a strip $\mathbb{R}\times [-1,1]$, then the associated shear flows for \eqref{NS} satisfy the boundary conditions and the constraint
\begin{equation}\label{BC}
\boldsymbol{u}=0\text{ at } y=\pm 1,~~ \ \ \ \int_{-1}^1 u_1(x, y) dy =\Phi.
\end{equation}
Here  the constraint comes from the divergence free condition of the velocity field and  $\Phi\in \mathbb{R}$ is called the flux of the flow.
When $\BF=0$,  the shear flows $\boldsymbol{U}=U(y)\Be_1$ for the problem \eqref{NS}-\eqref{BC} have explicit forms as follows
\begin{equation}\label{Poiseuille}
U(y)=\frac34\Phi(1-y^2),
\end{equation}
 The flow $\boldsymbol{U}$ is called the Poiseuille flow. Without loss of generality, the flux  $\Phi$ is assumed to be nonnegative.

The major breakthrough for the Leray problem was made by Amick \cite{A1,A2,AF} and Ladyzhenskaya and Solonnikov \cite{LS}. It was proved in \cite{A1, LS} that there is a unique solution for the Leray problem as long as the flux is small. The convergence rates of the solutions for the Navier-Stokes system with small flux in a channel was studied in \cite{AP,Horgan, HW, KP,Galdi,Morimoto, MF, NP1,NP2, NP3} and references therein. A significant open problem posed in \cite{Galdi} is to prove the existence of solutions for Leray problem when the flux is large. In fact,  it was proved in \cite{LS} that there exists a solution with arbitrary flux of the steady Navier-Stokes system in an infinitely long channel. Therefore, in order to solve the Leray problem, one needs only to show that the solutions obtained in \cite{LS} tend to Poiseuille flows at far fields.   To the best of our knowledge, there is no result on the far field behavior of steady solutions with large flux of Navier-Stokes system in a channel except for the axisymmetric solutions in a pipe studied in \cite{WX2}.   With the aid of the local compactness of the solutions and blowup techniques,  the key ingredient to get the far field behavior for the solutions obtained in \cite{LS} is to prove a Liouville type theorem for the Poiseuille flows in a strip. This is equivalent to the global uniqueness of Poiseuille flows in an infinitely long strip. However,  there are some essential difficulties to get even the local uniqueness of Poiseuille flows in a two dimensional infinitely long strip.
This is also quite different from the axisymmetric flows in a pipe where the local uniqueness and even the well-posedness with large external force have been established in \cite{WX1,WX2}. When the flow is symmetric in the strip, the existence of steady solutions for Navier-Stokes system with large flux was established in \cite{Rabier1}. The existence of general solutions in a strip was obtained in \cite{Rabier2} as long as the flux is not large.

In this paper, we consider the system \eqref{NS} in a strip $\Omega= \mathbb{T}_{2L\pi} \times[-1,1]$ which is periodic in $x$-direction with period $2 L \pi >0$. When supplemented with the no slip boundary condition and flux constraint \eqref{BC}, a natural question is whether the Poiseuille flows are unique in their uniform neighborhood  even when the flux is arbitrarily large. Let $\Bv=\Bu-\boldsymbol{U}$ be  the perturbation around the Poiseuille flow. It  satisfies the following system
\begin{equation}\label{model11}
\left\{\begin{aligned}
&-\Delta \Bv+\boldsymbol{U}\cdot \nabla \Bv+\Bv\cdot \nabla \boldsymbol{U}+\nabla P=-\Bv\cdot\nabla\Bv+\BF,\\
&\mathrm{div}~\Bv=0.
\end{aligned}\right.
\end{equation}
supplemented with the no-slip boundary conditions and flux constraint
\begin{equation}\label{model11'}
\Bv=0~~\text{on}~\partial\Omega,~~~~\ \ \ \ \ \int_{-1}^1v_1(x,y)dy=0.
\end{equation}
The crucial point for the analysis on the local uniqueness of the solutions for the problem \eqref{model11}-\eqref{model11'} is to study the associated linearized problem, i.e.,  the linear system
\begin{equation} \label{model12}
\left\{ \begin{aligned}
&- \Delta \Bv+\boldsymbol{U} \cdot \nabla \Bv + \Bv \cdot \nabla \boldsymbol{U}  + \nabla P = \BF, \ \ \ \mbox{in}\ \Omega,\\
& {\rm div}~\Bv = 0,
\end{aligned}\right.
\end{equation}
with the no-slip boundary conditions and flux constraint \eqref{model11'}.

The first main result of this paper can be stated as follow.
\begin{theorem}\label{thm1}
Assume that $\BF=\BF(x,y)\in L^2(\Omega)$. There exists a positive constant $L_0$ such that for any $L\leq L_0$, the linearized problem \eqref{model12} and \eqref{model11'} admits a unique periodic solution $\Bv(x,y)\in H^2(\Omega)$ satisfying
\begin{equation}
\|\Bv\|_{H^\frac53(\Omega)}\le C\|\BF\|_{L^2(\Omega)}
\label{estimate11}
\end{equation}
and
\begin{equation}\nonumber
\|\Bv\|_{H^2(\Omega)}\le C(1+\Phi^\frac14)\|\BF\|_{L^2(\Omega)},
\end{equation}
where $C$ is a uniform constant independent of the flux $\Phi$, $L$ and $\BF$.
\end{theorem}

There are a few remarks in order.
\begin{remark}
In fact, the bound $L_0$ is not very small (cf. \eqref{2-2-7-0} and \eqref{5-3-50} where the upper bound of $L$ is needed) and independent of the magnitude of the flux.
\end{remark}

\begin{remark}
The estimate \eqref{estimate11} is independent of the flux although the coefficients of the system \eqref{model12} depends on the flux. This is the key point to get the uniform nonlinear structural stability of the Poiseuille flows.
\end{remark}

\begin{remark}
When using the normal mode analysis to study the hydrodynamic stability of Poiseuille flows (\cite{L1,Orszag,DR}), the key issue to study the spectrum problem
\begin{equation} \label{spectralpb}
\left\{ \begin{aligned}
&s\Bv- \Delta \Bv+\boldsymbol{U} \cdot \nabla \Bv + \Bv \cdot \nabla \boldsymbol{U}  + \nabla P = \BF, \ \ \ \mbox{in}\ \Omega,\\
& {\rm div}~\Bv = 0,
\end{aligned}\right.
\end{equation}
supplemented with the boundary conditions \eqref{model11'}. Obviously, the system \eqref{model12} corresponds to the system \eqref{spectralpb} with $s=0$.
 Recently, it was proved in \cite{GGN} that a general class of symmetric shear flows  of the two dimensional incompressible Navier-Stokes equations in a periodic strip are spectrally unstable when the Reynolds number is sufficiently large. However, the detailed spectra of the linear problem is still not clear.
Theorem \ref{thm1} asserts  that $0$ cannot be the spectrum of the linear problem \eqref{spectralpb} and \eqref{model11'} for the flows with any Reynolds number.
This provides more details for the spectral set of the associated linearized operator. The stability or enhanced dissipations for plane shear flows in a strip under Navier slip boundary conditions  or in the whole plane was studied in \cite{Ding1, Ding2, Conti} .
\end{remark}

%According to the analysis to the linearized problem \eqref{model12}-\eqref{model12'}, one has that the $0$ mode of the $y$-direction velocity equals to zero. This simplifies the structure of the nonlinear system and provides a better estimate of the nonlinear term $\Bv\cdot \nabla \Bv$.

Based on Theorem \ref{thm1}, we prove the well-posedness theory of the nonlinear problem \eqref{model11}-\eqref{model11'}.

\begin{theorem}\label{largeforce}
Assume that $\BF=\BF(x,y)\in L^2(\Omega)$.
\begin{enumerate}
\item[(a)]
There exist constants $\varepsilon$ and $L_0$, such that for all  $L\leq L_0$, if
\begin{equation}\nonumber
\|\BF\|_{L^2(\Omega)}\le \varepsilon,
\end{equation}
the steady Navier-Stokes system \eqref{NS} supplemented with the no slip boundary condition and flux constraint \eqref{BC} admits a unique periodic solution $\Bu(x,y)$ satisfying
\begin{equation}\label{est1}
\left\|\Bu-\BU\right\|_{H^\frac53(\Omega)}\le C\|\BF\|_{L^2(\Omega)},
\end{equation}
and
\begin{equation}\label{est2}
\|\Bu-\BU\|_{H^2(\Omega)}\le C (1 + \Phi^\frac14 ) \|\BF\|_{L^2(\Omega)},
\end{equation}
where $C$ is a uniform constant independent of the flux $\Phi$, $L$, and $\BF$.
\item[(b)] There exist constants $\Phi_0$ and $L_0$, such that for all $\Phi\ge \Phi_0$ and $L\leq L_0$, if
\begin{equation}\nonumber
\|\BF\|_{L^2(\Omega)}\le \Phi^\frac{1}{32},
\end{equation}
the steady Navier-Stokes system \eqref{NS} supplemented with the no slip boundary condition and flux constraint \eqref{BC} admits a unique periodic solution $\Bu(x,y)$ satisfying \eqref{est1}-\eqref{est2} and
\[
\|u_2\|_{L^2(\Omega)} \leq \Phi^{-\frac{7}{12}},
\]
 where $C$ is a uniform constant independent of the flux $\Phi$, $L$, and $\BF$.
\end{enumerate}
\end{theorem}

\begin{remark}
Theorem \ref{largeforce} asserts that there exists a unique large solution in a suitable class of functions even when the external force is large.
\end{remark}

In fact, we have the following further result on the uniqueness of the solutions for the Navier-Stokes system in a strip.
\begin{theorem}\label{uniqueness}
Assume that $\BF=0$.
\begin{enumerate}
\item[(a)]
There exist constants $\varepsilon_1$ and $L_0$, such that for all  $L\leq L_0$, if $\Bu$ is a solution of the steady Navier-Stokes system \eqref{NS} in $\Omega$ supplemented the boundary conditions \eqref{BC} satisfying
\begin{equation}\label{est6-0}
\|u_2\|_{H^1(\Omega)}\le \varepsilon_1,
\end{equation}
then $\Bu\equiv\BU$, where $\BU=U(y)\Be_1$ is the Poiseuille flow with $U(y)$ defined in \eqref{Poiseuille}.
\item[(b)] There exist constants $L_0$ and $\Phi_0$, such that for all  $L\leq L_0$ and $\Phi\ge \Phi_0$, let $\Bu$ be a periodic solution of the problem \eqref{NS}-\eqref{BC} satisfying
\begin{equation}\label{est6-1}
\|u_2\|_{H^1(\Omega)}\le \Phi^{\frac{1}{60}},
\end{equation}
then $\Bu\equiv\BU$, where $\BU=U(y)\Be_1$ is the Poiseuille flow with $U(y)$ defined in \eqref{Poiseuille}.
\end{enumerate}
\end{theorem}
\begin{remark}
The uniqueness obtained in Theorem \ref{uniqueness} does not require any assumption on $u_1$. In order to prove the global unqiueness of Poiseuille flows, we need only remove the smallness assumption on $u_2$ later on.
\end{remark}
\begin{remark}
It the assumption  \eqref{est6-1} is replaced by
\begin{equation}\label{est6-3}
\|u_1-U(y)\|_{H^1(\Omega)}\leq \Phi^{\frac{1}{60}},
\end{equation}
then one can also prove the same conclusion in Theorem \ref{uniqueness}. However, the estimates obtained in Theorem \ref{largeforce} show that the vertical velocity $u_2$ decays as $\Phi$ goes to $\infty$. Hence the assumption \eqref{est6-1} seems to be more reasonable rather than \eqref{est6-3}.
\end{remark}

This paper is organized as follows. In Section \ref{Linear}, the stream function formulation for the linearized problem \eqref{model12} and \eqref{model11'}  is established and some basic a priori estimates for the stream function are given when the period $2\pi L$ is small. The existence and regularity of the solutions of the linearized problem are proved in Section \ref{sec-ex}. The  Section \ref{sec-res} devotes to establishing the  uniform a priori estimates independent of the flux $\Phi$ for the linearized problem by using boundary layer analysis. Then, the uniform nonlinear structural stability of Poiseuille flows is established in Section \ref{sec-nonlinear} with the help of the analysis on the associated linearized problem and a fixed point theorem. The well-posedness theory of the perturbed problem \eqref{model11}-\eqref{model11'} in the case for the large external force $\BF$ is also proved in Section \ref{sec-nonlinear}. The uniqueness of the solutions (Theorem \ref{uniqueness}) is proved in Section \ref{sec-unique}.  Some important lemmas which are used here and there in the paper are collected in appendix.

%%%%%%%%%%%%%%%%%%%%%%%%%%%%%%%%%%%Section2%%%%%%%%%%%%%%%%%%%%%%%%%%%%%%%%%%%%%%%%%%%%%%%%%%%%%%%%%%%%%%%%%%%%%%%%%%%%%%%%%%%

\section{Stream function formulation and a priori estimate}\label{Linear}
This section devotes to the basic a priori estimate for the linearized problem \eqref{model12} and \eqref{model11'} for periodic solutions. After introducing the stream function and its representation in terms of  Fourier series, the associated linearized problem is reduced into a sequence of boundary value problems of  the fourth order ODEs. The careful energy estimates for both the imaginary and the real parts of the complex ODEs give a good estimate for the solutions when the period is small. Although these estimates are not uniform with respect to the fluxes, they are enough to get the existence of solutions for the associated linear problem.

\subsection{Stream function formulation}\label{sec-stream}
For ease of notation,  for any $n\in \mathbb{Z}$, we denote $\hat{n}=\frac{n}{L}$ in the rest of the paper.
If the velocity field $\Bv$ is periodic with period $2\pi L$ in $x$-direction, then it can be written as
\begin{equation}\nonumber
\Bv=v_1(x,y)\Be_1 +v_2(x,y)\Be_2= \sum\limits_{n\in\mathbb{Z}}v_{1,n}(y)e^{i\hat{n}x}\Be_1+v_{2,n}(y)e^{i\hat{n}x}\Be_2,
\end{equation}
where
\begin{equation}\nonumber\begin{aligned}
(v_{1, n},v_{2, n})(y) &:=\frac{1}{2L\pi} \int_{-L\pi}^{L\pi} (v_1, v_2)(x,y) e^{-i\hat{n}x} \,dx.
\end{aligned}\end{equation}
Similarly, the $n$-th  mode of $\BF$ is denoted by
\begin{equation}\nonumber
\BF_n:= F_{1,n}(y)e^{i\hat{n}x}\Be_1+F_{2,n}(y)e^{i\hat{n}x}\Be_2.
\end{equation}
Since the velocity field $\Bv$ and force $\BF$ are periodic, $\nabla P$ must be periodic. Hence, one can write
\begin{equation}\nonumber
\partial_xP=\sum\limits_{n\in \Z}P_{1,n}(y)e^{i\hat{n}x}
\quad \text{and}
\quad
\partial_yP=\sum\limits_{n\in \Z}P_{2,n}(y)e^{i\hat{n}x}.
\end{equation}
Clearly, $P_{1,n}$ and $P_{2,n}$ satisfy
\begin{equation}\nonumber
i\hat{n}P_{2,n}=P_{1,n}'.
\end{equation}
Then the system \eqref{model12} becomes
\be \label{2-1-1}
\left\{\begin{aligned}
&i\hat{n}U(y)v_{1,n}+U'(y)v_{2,n}+P_{1,n}=\left(\frac{d^2}{d y^2}- {\hat{n}^2} \right)v_{1,n}+F_{1,n},\\
&i\hat{n}U(y)v_{2,n}+P_{2,n}=\left(\frac{d^2}{d y^2}- \hat{n}^2 \right)v_{2,n}+F_{2,n},\\
&\frac{d}{d y}v_{2,n}+i\hat{n}v_{1,n}=0.
\end{aligned}\right.
\ee
The  boundary condition and the flux constraint \eqref{model11'} can be written as
\be \label{2-1-2}
v_{2,n}(\pm1)  = v_{1,n}(\pm1) = 0,\quad  \int_{-1}^1 v_{1,n}(y) dy = 0.
\ee
Define
\begin{equation}\nonumber
\psi_n(y)=\left\{
\begin{aligned}
& \psi_0(-1)-\int_{-1}^y v_{1,0}(s)ds,\quad &\text{if}\,\, n= 0,\\
& -i\frac{1}{\hat{n}}v_{2,n},\quad &\text{if}\,\, n\neq 0,
\end{aligned}
\right.
\end{equation}
where $\psi_0(-1)$ is to be determined.
Then the vorticities of $\Bv_n$ and $\BF_n$ can be written as
\be \nonumber
\omega_n= i \hat{n} v_{2,n}-\frac{d}{d y}v_{1,n}=\left(\frac{d^2}{d y^2}- \hat{n}^2 \right)\psi_n \ \ \ \ \mbox{and}\ \ \ \ f_n= i\hat{n} F_{2,n}-\frac{d}{d y}F_{1,n},
\ee
respectively. It follows from \eqref{2-1-1} that $\omega_n$ satisfies
\begin{equation}\nonumber
 U''(y)v_{2,n} +i\hat{n} U(y)\omega_n -\left(\frac{d^2}{d y^2}-\hat{n}^2\right)\omega_n =f_n.
\end{equation}
Therefore, one has the following equation for $\psi_n$,
\be \label{stream}
-i\hat{n}U''(y)\psi_n+i \hat{n}U(y)\left(\frac{d^2}{d y^2}-\hat{n}^2 \right)\psi_n-\left(\frac{d^2}{d y^2}-\hat{n}^2\right)^2\psi_n= f_n.
\ee

Next, the boundary conditions \eqref{2-1-2} for $v_{1,n}$ and $v_{2,n}$ are equivalent to
\begin{equation*}\label{streamBC0}
\psi_0(1)=\psi_0(-1), \quad \psi_0'(\pm1)=0,
\quad \text{and}\quad \psi_n'(\pm1)=\psi_n(\pm1)=0\quad \text{for }n\in \mathbb{Z}, \, n\neq0.
\end{equation*}
Obviously, when $n=0$, the equation \eqref{stream} does not depend on $\psi_0$ itself and the velocity $v_{1,0}$ is the derivative of $\psi_0$. Therefore, without loss of generality, we assume that $\psi_0(1)=\psi_0(-1)=0$ so that the boundary conditions for $\psi_n$ can be written as
\begin{equation}\label{streamBC}
\psi_n'(\pm1)=\psi_n(\pm1)=0\quad \text{for }n\in \mathbb{Z}.
\end{equation}

%\newpage

%%%%%%%%%%%%%%%%%%%%%%%%%%%%%%%%%%%Low Order Estimates%%%%%%%%%%%%%%%%%%%%%%
\subsection{ A priori estimates for the stream function}\label{sec-apri}
In this subsection, some a priori estimates for the linear problem \eqref{stream} and \eqref{streamBC} are established, which guarantee the existence of solutions.
 The estimates consist in the following two lemmas.
\begin{lemma}\label{basic-est1}
Let $\psi_n(y)$ be a smooth solution of the problem \eqref{stream} and \eqref{streamBC}. There exists a constant $ L _0$, such that for all $ L \leq  L _0$, one has
\be \label{est2-1}
\int_{-1}^1 |\psi_n''|^2 +  \hat{n}^2 \left| \psi_n' \right|^2 +  \hat{n}^4 | \psi_n|^2 \, dy
\leq C \int_{-1}^1 |f_n|^2  \, dy,
\ee
where $C$ is a  uniform constant independent of $f_n$, $n$, and $\Phi$.
\end{lemma}

\begin{proof} Multiplying \eqref{stream} by the complex conjugate of $\psi_n$ and integrating the corresponding equation over $[-1,1]$ yield
\begin{equation}\label{2-2-1}
\int_{-1}^1 \left[-i \hat{n}  U''\psi_n+ i\hat{n} U\left(\frac{d^2}{d y^2}- \hat{n}^2 \right)\psi_n-\left(\frac{d^2}{d y^2}- \hat{n}^2\right)^2\psi_n\right] \overline{\psi_n}\,d y=\int_{-1}^1 f_n\overline{\psi_n}\,dy.
\end{equation}
For the last two terms on the left hand side of \eqref{2-2-1}, it follows from integration by parts and the homogeneous boundary conditions \eqref{streamBC} that one has
\begin{equation}\nonumber
\begin{aligned}
\int_{-1}^1 i\hat{n} U \left(\frac{d^2}{d y^2}- \hat{n}^2\right)\psi_n \overline{\psi_n}\,dy=&\int_{-1}^1-i\hat{n}^3U|\psi_n|^2 -i\hat{n} U|\psi_n'|^2-i\hat{n} U'\psi_n'\overline{\psi_n}\,dy
\end{aligned}
\end{equation}
and
\begin{equation}\nonumber
\begin{aligned}
\int_{-1}^1-\left(\frac{d^2}{d y^2}- \hat{n}^2\right)^2\psi_n \overline{\psi_n}\,dy=&\int_{-1}^1- \hat{n}^4 |\psi_n|^2 +2 \hat{n}^2 \psi_n''\overline{\psi_n}-\psi_n^{(4)}\overline{\psi_n}\,dy\\
=&\int_{-1}^1- \hat{n}^4 |\psi_n|^2-2 \hat{n}^2|\psi_n'|^2-|\psi''_n|^2\,dy.
\end{aligned}
\end{equation}
Then we rewrite the imaginary and real part of \eqref{2-2-1} as
\begin{equation}\label{2-2-4}
\int_{-1}^1  \hat{n}^4 |\psi_n|^2+2 \hat{n}^2|\psi_n'|^2+|\psi''_n|^2\,d y=-\Re\int_{-1}^1 f_n\overline{\psi_n}\,dy+\Im\int_{-1}^1 \hat{n} U'\psi_n'\overline{\psi_n}\,dy
\end{equation}
and
\begin{equation}\label{2-2-5}
\int_{-1}^1 \hat{n}U''|\psi_n|^2+ \hat{n}^3U|\psi_n|^2 + \hat{n}U|\psi_n'|^2\,d y=-\Im\int_{-1}^1 f_n\overline{\psi_n}\,dy-\Re\int_{-1}^1 \hat{n}U'\psi_n'\overline{\psi_n}\,dy,
\end{equation}
respectively. Note that
\begin{equation}\nonumber
\Re\int_{-1}^1 \hat{n}U'\psi'_n\overline{\psi_n}\,dy
=\frac12\int_{-1}^1 \hat{n}U'\left(\overline{\psi'_n}\psi_n+\psi_n'\overline{\psi_n}\right)\,dy=-\frac12\int_{-1}^1 \hat{n}U''|\psi_n|^2\,dy.
\end{equation}
Hence the equation \eqref{2-2-5} can be rewritten as follows,
\begin{equation}\label{2-2-6}
\frac{3\Phi\hat{n}}{4}\int_{-1}^1\hat{n}^2|\psi_n|^2(1-y^2)+ |\psi_n'|^2(1-y^2)\,d y=-\Im\int_{-1}^1 f_n\overline{\psi_n}\,dy+\frac{3\Phi\hat{n}}{4}\int_{-1}^1 |\psi_n|^2\, dy.
\end{equation}

If $|n|\ge 1$, according to Lemma \ref{lemmaHLP}, there exists a constant $ L _0$ such that for all $ L \le  L _0$, one has
\begin{equation}\label{2-2-7-0}
\int_{-1}^1 |\psi_n|^2\,dy \leq \frac13\int_{-1}^1\hat{n}^2|\psi_n|^2(1-y^2)+ |\psi_n'|^2(1-y^2)\,d y.
\end{equation}
The inequality \eqref{2-2-6}, together with \eqref{2-2-7-0}, yields that
\begin{equation}\label{2-2-7}
|\Phi\hat{n}|\int_{-1}^1|\psi_n|^2\,dy\le C \left| \int_{-1}^1 f_n\overline{\psi_n}\,dy \right| .
\end{equation}
Hence
\begin{equation}\label{2-2-8}
|\Phi\hat{n}|^2\int_{-1}^1|\psi_n|^2\,dy\le C\int_{-1}^1|f_n|^2\,dy.
\end{equation}
If $n=0$, the estimate \eqref{2-2-8} clearly holds.

The estimate \eqref{2-2-4}, together with \eqref{2-2-8}, Lemma \ref{lemmaA1} and Young's inequality, yields that
\begin{equation}\nonumber
\begin{aligned}
&\int_{-1}^1 \hat{n}^4 |\psi_n|^2+ \hat{n}^2|\psi_n'|^2+|\psi''_n|^2\,d y
\le~~ C\int_{-1}^1 |f_n\overline{\psi_n}|\,dy+C|\Phi\hat{n}|\int_{-1}^1 |\psi_n'\overline{\psi}_n|\,dy\\
\leq~~& \frac12\int_{-1}^1 |\psi_n''|^2\,dy+C\int_{-1}^1 |f_n|^2\,dy+C|\Phi\hat{n}|^2\int_{-1}^1 |\psi_n|^2\,dy\\
\leq~~& \frac12\int_{-1}^1 |\psi_n''|^2\,dy+C\int_{-1}^1 |f_n|^2\,dy.
\end{aligned}
\end{equation}
Thus one has \eqref{est2-1} and finishes the proof of the lemma.
\end{proof}

%%%%%%%%%%%%%%%%%%%%%%%%%%%%High Order Estimates%%%%%%%%%%%%%%%%%%%%%%%%%%%%%%

Using the similar idea as in the proof of  Lemma \ref{basic-est1}, one has the following  higher order a priori estimates.
\begin{lemma}\label{basic-est2}
Let $\psi_n(y)$ be a smooth solution of the problem \eqref{stream} and \eqref{streamBC}. For all $ L \leq  L _0$, one has
\begin{equation}\label{est2-2}
\inte  \hat{n}^4   | \psi_n''|^2
+ \hat{n}^6  | \psi_n'|^2+\hat{n}^8 |  \psi_n|^2  \, dy
\leq C (1+\Phi) \inte |f_n|^2 \,dy
\end{equation}
and
\begin{equation}\label{est2-3}
\inte |\psi_n^{(4)} |^2  \, dy \leq C (1+ \Phi^2) \inte |f_n|^2\,dy,
\end{equation}
where $C$ is a  uniform constant independent of $f_n$, $n$ and $\Phi$.
\end{lemma}

\begin{proof}
 Multiplying \eqref{2-2-4} by $ \hat{n}^2$ gives
\begin{equation}\nonumber
\begin{aligned}
&\inte  \hat{n}^2|\psi_n''|^2   + 2 \hat{n}^4 |\psi_n'|^2+\hat{n}^6|\psi_n|^2 \,dy
\leq \,\,   \hat{n}^2 \int_{-1}^1 |f_n\overline{\psi_n}| \, dy + C \Phi |\hat{n}|^3\int_{-1}^1 |\psi_n'\overline{\psi_n}| \, dy \\
\leq \,\,&   \hat{n}^2 \int_{-1}^1 |f_n\overline{\psi_n}| \, dy +  \hat{n}^4 \int_{-1}^1 |\psi_n'|^2\, dy+C|\Phi\hat{n}|^2\int_{-1}^1 |\psi_n|^2 \, dy.
\end{aligned}
\end{equation}
It follows from Lemma \ref{lemmaA1},  Young's inequality, \eqref{est2-1} and \eqref{2-2-8} that one has
\begin{equation} \label{2-2-11}
\inte  \hat{n}^2|\psi_n''|^2+  \hat{n}^4 |\psi_n'|^2+\hat{n}^6|\psi_n|^2 \,dy\leq C\int_{-1}^1 |f_n|^2 \, dy .
\end{equation}

Similarly, multiplying \eqref{2-2-4} by $ \hat{n}^4 $ yields
\be \nonumber
\ba
&\inte  \hat{n}^4 |\psi_n''|^2   + 2\hat{n}^6|\psi_n'|^2+\hat{n}^8|\psi_n|^2 \,dy
\leq~~   \hat{n}^4  \int_{-1}^1 |f_n\overline{\psi_n}| \, dy + C\Phi |\hat{n}|^5\int_{-1}^1 |\psi_n'\overline{\psi_n}| \, dy \\
\leq~~&  \hat{n}^4  \int_{-1}^1 |f_n\overline{\psi_n}| \, dy  + C\Phi \hat{n}^4\int_{-1}^1 | \psi_n'|^2  \, dy
+ C \Phi\hat{n}^6 \int_{-1}^1 |\psi_n|^2  \, dy.
\ea \ee
By Young's inequality and the inequality \eqref{2-2-11}, one has
\be \nonumber
\inte  \hat{n}^4 |\psi_n''|^2+\hat{n}^6|\psi_n'|^2+\hat{n}^8|\psi_n|^2 \,dy
\leq~~C(1+\Phi)\int_{-1}^1 |f_n|^2\, dy,
\ee
which gives exactly the estimate \eqref{est2-2}.

Next, to get the high order regularity of $\psi_n$, one can write the equation \eqref{stream} as
\be\label{2-2-14}
\psi_n^{(4)}=-i\hat{n} U''\psi_n+i\hat{n} U\psi_n''-i\hat{n}^3U\psi_n+2 \hat{n}^2\psi_n''- \hat{n}^4 \psi_n-f_n.
\ee
Multiplying \eqref{2-2-14} by $\overline{\psi^{(4)}_n}$ and integrating over $[-1, 1]$ yield
\be\label{2-2-15}
\inte\left|\psi_n^{(4)}\right|^2\,dy=\inte\left[-i\hat{n} U''\psi_n+i\hat{n} U\psi_n''-i\hat{n}^3U\psi_n+2 \hat{n}^2\psi_n''- \hat{n}^4 \psi_n-f_n\right]\overline{\psi^{(4)}_n}\,dy.
\ee
Using Cauchy-Schwarz inequality and \eqref{2-2-8} gives
\be \label{2-2-16}
\ba
& \left|\inte i\hat{n} U''\psi_n\overline{\psi^{(4)}_n}\,dy\right|\le C|\Phi\hat{n}|\inte\psi_n\overline{\psi^{(4)}_n}\,dy\\
\leq~~&\frac18\inte\left|\psi^{(4)}_n\right|^2\,dy+C|\Phi\hat{n}|^2\inte|\psi_n|^2\,dy\\
\leq~~&\frac18\inte\left|\psi^{(4)}_n\right|^2\,dy+C\inte|f_n|^2\,dy.
\ea
\ee
By the  estimates \eqref{2-2-11}, one has
\be \label{2-2-17}
\ba
& \,\,\left|\int_{-1}^1i\hat{n}U \psi_n''  \overline{\psi^{(4)}_n} \, dy  \right|
\leq C|\Phi\hat{n}|\int_{-1}^1 \left|\psi_n''\overline{\psi^{(4)}_n}\right| \, dy \\
\leq & \, \, C|\Phi\hat{n}|^2 \int_{-1}^1 |\psi_n''|^2  \, dy +  \frac18 \int_{-1}^1 | \psi^{(4)}_n|^2   \, dy \\
 \leq &\,\, C  \Phi^2  \int_{-1}^1 |f_n|^2   \, dy +  \frac18 \int_{-1}^1 \left|\psi^{(4)}_n\right|^2   \, dy
\ea \ee
and
\be \label{2-2-18}
\ba
&\,\, \left|  \int_{-1}^1i\hat{n}^3 U \psi_n \overline{\psi^{(4)}_n}   \, dy   \right|  \leq C \Phi^2\hat{n}^6 \int_{-1}^1 |\psi_n|^2   \, dy + \frac18 \int_{-1}^1 \left|\psi^{(4)}_n\right|^2    \, dy \\
 \leq &\,\, C \Phi^2   \int_{-1}^1 |f_n|^2  \, dy +  \frac18 \int_{-1}^1 \left|\psi^{(4)}_n\right|^2  \, dy.
\ea \ee
Similarly, using estimate \eqref{est2-2} gives
\be \label{2-2-19}
\ba
& \,\,\left|   \int_{-1}^1 2 \hat{n}^2 \psi_n'' \overline{\psi_n^{(4)}}  \, dy  \right|
 \leq C  \hat{n}^4  \int_{-1}^1 | \psi_n''|^2   \, dy + \frac18 \int_{-1}^1\left|\psi_n^{(4)}\right|^2   \, dy \\
 \leq &\,\, C (1+ \Phi)  \int_{-1}^1 |f_n|^2   \, dy +  \frac18 \int_{-1}^1 \left|\psi_n^{(4)}\right|^2 \, dy
\ea\ee
and
\be \label{2-2-20}
\ba
&\,\, \left| \int_{-1}^1  \hat{n}^4 \psi_n  \overline{\psi_n^{(4)}}   \, dy   \right|
 \leq  C \hat{n}^8 \int_{-1}^1 | \psi_n|^2   \, dy + \frac18 \int_{-1}^1\left|\psi_n^{(4)}\right|^2   \, dy \\
 \leq &\,\, C (1+ \Phi)  \int_{-1}^1 |f_n|^2   \, dy +  \frac18 \int_{-1}^1 \left|\psi_n^{(4)}\right|^2 \, dy.
\ea \ee
Hence, combining all the estimates \eqref{2-2-15}--\eqref{2-2-20} gives \eqref{est2-3}. This finishes the proof of the lemma.
\end{proof}

%%%%%%%%%%%%%%%%%%%%%%%%%%%%%Existence and Regulaity for Linear Problem%%%%%%%%%%%%%%%%%%%%%%%%
%\newpage

\section{Existence and regularity of solutions for the linearized problem}\label{sec-ex}
In this section, the existence and regularity of the solutions for the linearized problem  \eqref{stream} and \eqref{streamBC} are established.
\subsection{Existence of solutions for the linearized problem}
In this subsection, the existence of solutions to the problem \eqref{stream} and \eqref{streamBC} for each fixed $n$ is established via the Galerkin method with the aid of the a priori estimates obtained in Lemmas \ref{basic-est1}-\ref{basic-est2}.

The following lemma gives the existence of an orthonormal basis of $L^2(-1, 1)$, which also belongs to $H_0^2(-1, 1) \cap H^4(-1, 1)$.
\begin{lemma}
There exists an orthonormal basis $\{\varphi_m\}\subset H^2_0(-1,1)\cap H^4(-1,1)$ of $L^2(-1,1)$.
\end{lemma}
\begin{proof}
Let us consider the following problem,
\be\label{model31}
\left\{\ba
&\varphi^{(4)}=g,\\
&\varphi(\pm1)=\varphi'(\pm1)=0,
\ea\right.\ee
where $g \in L^2(-1,1)$. Applying Lax-Milgram theorem shows that the problem \eqref{model31} admits a unique solution in $H^2_0(-1,1)$. Indeed, for any two functions $\varphi,\phi\in H_0^2(-1,1)$, define the bilinear functional
\[ \mcL (\varphi,\phi):=\inte \varphi''\overline{\phi''}\,dy
\]
and the linear functional
\[ \mcG (\phi):=\inte g\overline{\phi}\,dy.
\]
For any $\phi\in H_0^2(-1,1)$,  Poincar\'e's inequality gives
\begin{equation}\nonumber
\mcL (\phi,\phi)=\inte |\phi''|^2\,dy\ge c_0\|\phi\|_{H_0^2(-1,1)}^2,
\end{equation}
where $c_0>0$ is a uniform constant. Moreover, it follows from H\"{o}lder's inequality that one has
\begin{equation}\nonumber
\mcL (\varphi,\phi)\le C\|\varphi\|_{H_0^2(-1,1)}\|\phi\|_{H_0^2(-1,1)}
\end{equation}
and
\begin{equation}\nonumber
\mcG (\phi)\le C\|\phi\|_{H_0^2(-1,1)}\|g\|_{L^2(-1,1)}.
\end{equation}
Hence, for any $g\in L^2(-1,1)$, the problem \eqref{model31} admits a unique solution $\varphi\in H_0^2(-1,1)$ which satisfies
\begin{equation}\nonumber
\|\varphi\|_{H_0^2(-1,1)}\le C\|g\|_{L^2(-1,1)}.
\end{equation}
Moreover, for any $g\in C^\infty([-1,1])$, the corresponding solution $\varphi\in C^\infty([-1,1])$ and satisfies
\begin{equation}\nonumber
\|\varphi\|_{H^4(-1,1)}\leq C\|g\|_{L^2(-1,1)}.
\end{equation}
Then by density argument, for $g\in L^2(-1,1)$, the problem \eqref{model31} admits a unique solution $\varphi\in H^2_0(-1,1)\cap H^4(-1,1)$ satisfying
\begin{equation}\nonumber
\|\varphi\|_{H^4(-1,1)}\leq C\|g\|_{L^2(-1,1)}.
\end{equation}
Next, define the solution
\[\varphi=\mathcal{M}g.\]
For any $g\in L^2(-1,1)$, one has
\be\nonumber
\inte \mathcal{M}g\overline{g}\,dy=\inte\varphi\overline{\varphi^{(4)}}\,dy =\inte\varphi^{(4)}\overline{\varphi}\,dy=\inte g\mathcal{M}\overline{g}\,dy.
\ee
Hence $\mathcal{M}$ is a compact symmetric operator on $L^2(-1,1)$. Owing to Hilbert-Schmidt theory (\cite{Lax}), the eigenfunctions of the operator $\mathcal{M}$ constitute an orthonormal basis of $L^2(-1,1)$.
\end{proof}

If $ L \le  L _0$, based on the a priori estimates obtained in Lemmas \ref{basic-est1}-\ref{basic-est2}, the existence of the solution to the problem \eqref{stream} and \eqref{streamBC} can be established via the standard Galerkin approximation method for every $n\in\mathbb{Z}$.
Since all the a priori estimates hold for the approximation solutions, they also hold for the
solution $\psi_n$. The uniqueness of the solution is  the direct consequence of the a priori estimates.

\begin{definition}\label{def1}
For every $n\in\mathbb{Z}$, $\varphi\in H^2_0(I)\cap H^4(I)$ with $I=(-1, 1)$, define
\begin{equation}\nonumber
\|\varphi\|_{H_{n}^4(I)}^2  : = \int_{-1}^1 |\varphi''|^2 +  \hat{n}^4  |\varphi|^2 \,dy  + \int_{-1}^1 |\varphi^{(4)}|^2 +  \hat{n}^4  |\varphi''|^2+ \hat{n}^8|\varphi|^2\, dy.
\end{equation}
\end{definition}

The existence result for $\psi_n$ follows from the a priori estimates established in Section \ref{Linear} .
\begin{pro} \label{existence-stream}
Assume that  $f_n(y)\in L^2(I)$. For any $ L \leq L _0$, there exists a unique solution $\psi_n \in H_{n}^4(I)$ to the linear problem  \eqref{stream}-\eqref{streamBC} which satisfies
\begin{equation}\nonumber
\|\psi_n\|_{H_{n}^4(I)} \leq C (1 + \Phi) \|f_n\|_{L^2(I)},
\end{equation}
where $C$ is a positive constant independent of $f_n$, $n$, $ L $, and $\Phi$.
\end{pro}

\subsection{Regularity of the velocity field} \label{sec-reg}
In this subsection, we investigate the properties of functions in
$H_{n}^4(I)$ and the regularity of $\Bv$.

\begin{lemma}\label{reg-stream}
Let $\varphi_n$ be a function in $H_{n}^4(I)$  defined in Definition \ref{def1}. Then there exists a positive constant $C$, independent of $n$, $ L $, and $\varphi_n$, such that
\be \nonumber
\ba
 \int_{-1}^1|\varphi_n |^2 +|\varphi_n'|^2 + \hat{n}^2 |\varphi_n |^2+\hat{n}^2|\varphi_n'|^2+ \left|\varphi_n^{(3)}\right|^2+ \hat{n}^2 |\varphi_n''|^2 \,dy&~~\\
+ \inte \hat{n}^4  \left| \varphi_n'\right|^2 + \hat{n}^6 | \varphi_n |^2 +\hat{n}^2\left| \varphi_n^{(3)}\right|^2 + \hat{n}^6 | \varphi_n'|^2\, dy&\leq C \|\varphi_n\|_{H_{n}^4(I)}^2.
\ea \ee
\end{lemma}

\begin{proof}For simplicity, assume that $\varphi_n \in C^\infty([-1,1])$ with $\varphi_n(\pm1)=\varphi_n'(\pm1)=0$. Following the proof
of Lemma \ref{lemmaA1} and using the homogeneous boundary conditions for $ \varphi_n$ give
\be \nonumber
\int_{-1}^1 \left| \varphi_n\right|^2  +\left| \varphi_n'\right|^2  \, dy\leq  C\int_{-1}^1 |\varphi_n''|^2  \, dy \leq C \|\varphi_n\|_{H_{n}^4(I)}^2.
\ee
Furthermore, it follows from  Cauchy-Schwarz inequality that one has
\be \nonumber
\begin{aligned}
&\int_{-1}^1  \hat{n}^2  |\varphi_n|^2 +  \hat{n}^6 |\varphi_n|^2 +  \hat{n}^2 |\varphi_n''|^2 \,dy \\
\leq&
\int_{-1}^1    |\varphi_n|^2 +  \hat{n}^4 |\varphi_n|^2+ \hat{n}^8 |\varphi_n|^2+ |\varphi_n''|^2+ \hat{n}^4  |\varphi_n''|^2 \,dy \leq C \|\varphi_n\|_{H_{n}^4(I)}^2.
\end{aligned}\ee

Using integration by parts and the homogeneous boundary conditions for $\varphi_n$ yields
\be \nonumber \ba
\int_{-1}^1 \hat{n}^2\left|\varphi_n' \right|^2\, dy
 =  -  \int_{-1}^1\hat{n}^2\varphi_n''\overline{\varphi_n}   \, dy\leq C\int_{-1}^1 |\varphi_n|^2  + \hat{n}^4  |\varphi_n''|^2   \, dy.
\ea \ee
This implies
\be \nonumber
 \int_{-1}^1  \hat{n}^2\left|\varphi_n' \right|^2\, dy
 \leq C \|\varphi_n\|_{H_{n}^4(I) }^2.
\ee

Similarly, one has
\be\nonumber
\int_{-1}^1 \hat{n}^4  \left|\varphi_n '\right|^2 +\hat{n}^6 \left|\varphi_n '\right|^2   \, dy
\leq C \|\varphi_n\|_{H_{n}^4(I)}^2.
\ee
Moreover, Lemma \ref{lemmaA2} gives
\be \label{4-1-6}
 \int_{-1}^1 \left|\varphi_n^{(3)}\right|^2\, dy
\leq C \int_{-1}^1 |\varphi_n''|^2  \, dy + C\left(\int_{-1}^1 |\varphi_n''|^2  \, dy\right)^\frac12\left(\int_{-1}^1\left|\varphi_n^{(4)}\right|^2  \, dy\right)^\frac12.
 \ee
Hence one has
\begin{equation}\nonumber
 \int_{-1}^1 \left|\varphi_n^{(3)}\right|^2\, dy
 \leq C\|\varphi_n\|_{H_{n}^4(I)}^2.
\end{equation}

Multiplying \eqref{4-1-6} by $ \hat{n}^2$ gives
\be \nonumber
\begin{aligned}
\int_{-1}^1 \hat{n}^2 \left|\varphi_n^{(3)}\right|^2\, dy
\leq &C\int_{-1}^1  \hat{n}^2|\varphi_n''|^2  \, dy + C \hat{n}^2\left(\int_{-1}^1 |\varphi_n''|^2  \, dy\right)^\frac12\left(\int_{-1}^1\left|\varphi_n^{(4)}\right|^2  \, dy\right)^\frac12\\
\leq \,\,& C\int_{-1}^1\left(1+ \hat{n}^4 \right)  |\varphi_n''|^2   \, dy
+ C\int_{-1}^1 \left|\varphi_n^{(4)}\right|^2   \, dy\\
\leq \,\,& C\|\varphi_n\|_{H_{n}^4(I)}^2.
\end{aligned}
\ee
This finishes the proof of the lemma.
\end{proof}

%%%%%%%%%%%%%%%%%%%%%%%%%%%%%%%%%%%%%%%%%%%Regularity%%%%%%%%%%%%%%%%%%%%%%%%
For any function $\psi_n \in H_{n}^4(I)$, let
\begin{equation}\nonumber
\Bv_n=v_{1,n}(y)e^{i\hat{n}x}\Be_1+v_{2,n}(y)e^{i\hat{n}x}\Be_2=-\psi_n'(y)e^{i\hat{n}x}\Be_1+i\hat{n} \psi_n(y) e^{i\hat{n}x}\Be_2
\end{equation}
and
\begin{equation}\nonumber
\Bo_n=\left(\frac{d^2}{dy^2}- \hat{n}^2\right)\psi_n(y)e^{i\hat{n}x}
\end{equation}
be the corresponding velocity field and vorticity, respectively. First, one has the following $H^3(\Omega)$-bound of $\Bv_n$.
\begin{lemma}\label{reg-velocity}
Assume that $\psi_n\in H_{n}^4(I)$. There exists a constant $C$, independent of $n$, $ L $, and $\psi_n$, such that
\begin{equation}\nonumber
\|\Bv_n\|_{H^3(\Omega)}\leq C  L ^\frac12\|\psi_n \|_{H_{n}^4(I)}.
\end{equation}
\end{lemma}

\begin{proof}
Note that $v_{2,n} =  i  n  \psi_n $ and $v_{1,n} = -\psi_n'$. By virtue of Lemma
\ref{reg-stream}, one has
\begin{equation}\nonumber
\begin{aligned}
\|\Bv_n\|_{L^2(\Omega)}^2=&\int_{-\pi L}^{\pi L}\int_{-1}^1 \left|\hat{n} \psi_n e^{i\hat{n}x}\right|^2 +\left| \psi_n'e^{i\hat{n}x}\right|^2 \, dydx\\
\leq& 2\pi L \int_{-1}^1  \hat{n}^2 |\psi_n|^2 +|\psi_n'|^2 \, dy
\leq 2\pi L \|\psi_n\|_{H_{n}^4(I) }^2
\end{aligned}
\end{equation}
and
\begin{equation}\nonumber
\ba
\left\|\Bo_n\right\|_{H^2(\Omega)}^2\leq&~~ C L \inte \left|\left( \frac{d^2}{dy^2}- \hat{n}^2\right)\psi_n\right|^2+\left|\frac{d}{dy}\left( \frac{d^2}{dy^2}- \hat{n}^2 \right)\psi_n\right|^2\,dy\\
&~~+C L \inte  \hat{n}^2\left|\left( \frac{d^2}{dy^2}- \hat{n}^2 \right)\psi_n\right|^2+\left|\frac{d^2}{dy^2}\left( \frac{d^2}{dy^2}- \hat{n}^2 \right)\psi_n\right|^2\,dy\\
&~~+C L \inte  \hat{n}^2\left|\frac{d}{dy}\left( \frac{d^2}{dy^2}- \hat{n}^2 \right)\psi_n\right|^2+ \hat{n}^4 \left|\left( \frac{d^2}{dy^2}- \hat{n}^2 \right)\psi_n\right|^2\,dy\\
\leq&~~C L \|\psi_n\|_{H_{n}^4(I)}^2.
\ea
\end{equation}

Moreover, the straightforward computations give that
\be \label{vorticity}
\Delta \Bv_n = \left(i\hat{n}\omega_n,-\frac{d}{dy}\omega_n\right)^Te^{i\hat{n}x}\ \ \ \ \mbox{in}\ \Omega.
\ee
Applying the homogeneous boundary conditions and the regularity theory (\cite[Theorem 8.12]{GT}) for the elliptic equation \eqref{vorticity} yields
\be \nonumber
\| \Bv_n \|_{H^3(\Omega) } \leq C \left\|\nabla(\omega_n e^{i\hat{n}x})\right\|_{H^1(\Omega)}+C\| \Bv_n \|_{L^2(\Omega) }\leq C  L ^\frac12 \|\psi_n\|_{H_{n}^4(I)}.
\ee
This completes the proof of Lemma \ref{reg-velocity}.
\end{proof}

The existence of solution of the problem  \eqref{model12} and \eqref{model11'} can be established.
\begin{pro}\label{back} For any $ L \leq  L _0$, if $\psi_n$ is the solution obtained in Proposition \ref{existence-stream}, then $\Bv=\sum\limits_{n\in\mathbb{Z}}\Bv_n$ is a strong solution to the problem \eqref{model12} and \eqref{model11'}  and satisfies
\be \nonumber
\|\Bv\|_{H^3(\Omega)} \leq C (1 + \Phi ) \| \BF\|_{H^1(\Omega)}.
\ee
\end{pro}

\begin{proof} Since $\psi_n$ is a solution to the problem \eqref{stream} and \eqref{streamBC}, direct computation yields
\be \nonumber
{\rm curl}~\left( (\boldsymbol{U}\cdot \nabla ) \Bv_n + (\Bv_n \cdot \nabla) \boldsymbol{U}  \right) - {\rm curl}~(\Delta \Bv_n) = {\rm curl}~\BF_n.
\ee
Hence, for each $n$ there exists some function $P_n$ with $\nabla P_n\in L^2(\Omega)$, such that
\be \nonumber
 (\boldsymbol{U}\cdot \nabla ) \Bv_n + (\Bv_n \cdot \nabla) \boldsymbol{U} - \Delta \Bv_n + \nabla P_n = \BF_n.
\ee
Therefore $\Bv=\sum\limits_{n\in\mathbb{Z}}\Bv_n$ is a strong solution to the problem \eqref{model12} and \eqref{model11'}.  According to Lemma \ref{reg-velocity} and Proposition \ref{existence-stream}, one has
\be \nonumber
\|\Bv_n\|_{H^3(\Omega)} \leq C  L ^\frac12 \|\psi_n\|_{H_{n}^4(I)}\leq C  L ^\frac12(1 + \Phi) \|f_n\|_{L^2(I)} \leq C (1 + \Phi) \|\BF_n\|_{H^1(\Omega)}.
\ee
This finishes the proof of the proposition.
\end{proof}

\section{Uniform estimate  with respect to the flux}\label{sec-res}
The goal of this section is to establish the uniform estimate with respect to the flux for the solution of the problem \eqref{model12} and \eqref{model11'}. We investigate the problem with three different cases based on the magnitude of the frequency.

\subsection{Estimate for the case with small flux}
In this subsection, the solutions of the problem \eqref{stream} and \eqref{streamBC} are estimated in terms of $\BF$ when the flux $\Phi$ is not large.
\begin{pro}\label{smallflux}
Let $\psi_n$ be the solution obtained in Proposition \ref{existence-stream}, the corresponding velocity field $\Bv_n$ satisfies
\be\nonumber
\|\Bv_n\|_{H^2(\Omega)} \leq C (1 + \Phi^2) \|\BF_n\|_{L^2(\Omega)}.
\ee
\end{pro}
\begin{proof}
It follows from \eqref{2-2-7} and integration by parts that
\begin{equation}\label{5-1-1}
\ba
|\Phi\hat{n}|\inte |\psi_n|^2\,dy\le&~~C\left|\inte i\hat{n} F_{2,n}\overline{\psi_n}-F_{1,n}\overline{\psi_n'}\,dy\right|\\
\le&~~C\left(\inte |\BF_{n}|^2\,dy\right)^\frac12\left(\inte  \hat{n}^2|\psi_n|^2 +\left|\psi'_n\right|^2\,dy\right)^\frac12.
\ea\end{equation}
Similarly, the estimate \eqref{2-2-4} yields
\begin{equation}\label{5-1-2}\ba
&~~\inte\left|\psi_n''\right|^2+2 \hat{n}^2\left|\psi_n'\right|^2+ \hat{n}^4 |\psi_n|^2\,dy\\
\le&~~C\left(\inte |\BF_{n}|^2\,dy\right)^\frac12\left(\inte  \hat{n}^2|\psi_n|^2 +\left|\psi'_n\right|^2\,dy\right)^\frac12+\frac32|\Phi\hat{n}|\inte \left|\psi'_n\right|\left|\psi_n\right|\,dy.
\ea\end{equation}
According to \eqref{5-1-1}, one has
\begin{equation}\label{5-1-3}\ba
&~~\frac32|\Phi\hat{n}|\inte \left|\psi_n'\right|\left|\psi_n\right|\,dy\leq C\Phi^\frac12\left(|\hat{n}|\inte \left|\psi_n'\right|^2 \,dy\right)^\frac12\left(|\Phi\hat{n}|\inte \left|\psi_n\right|^2 \,dy\right)^\frac12\\
\leq~~&C\Phi^\frac12\left(\inte |\BF_{n}|^2\,dy\right)^\frac14\left(|\hat{n}|\inte \left|\psi_n'\right|^2 \,dy\right)^\frac12\left(\inte  \hat{n}^2|\psi_n|^2 +\left|\psi'_n\right|^2\,dy\right)^\frac14.
\ea\end{equation}
Therefore, it follows from \eqref{5-1-2}-\eqref{5-1-3}, Young's inequality and Lemma \ref{lemmaA1} that
\begin{equation}\nonumber
\inte\left|\psi_n''\right|^2+ \hat{n}^2\left|\psi_n'\right|^2+ \hat{n}^4 |\psi_n|^2\,dy\leq C(1+\Phi^2)\inte |\BF_{n}|^2\,dy.
\end{equation}
This implies
\begin{equation}\nonumber
\|\nabla\Bv_n\|_{L^2(\Omega)}\leq C(1+\Phi)\|\BF_n\|_{L^2(\Omega)}.
\end{equation}
By Poinc\'{a}re inequality, one has
\begin{equation}\nonumber
\|\Bv_n\|_{H^1(\Omega)}\leq C(1+\Phi)\|\BF_n\|_{L^2(\Omega)}.
\end{equation}

%Multiplying \eqref{5-1-2} and \eqref{5-1-3} by $ \hat{n}^2$, respectively, one has
%\begin{equation}\label{5-1-7}\ba
%&\inte  \hat{n}^2\left|\psi_n''\right|^2+2  \hat{n}^4 \left|\psi_n'\right|^2+\hat{n}^6|\psi_n|^2\,dy\\
%\leq& C\left(\inte |\BF_{n}|^2\,dy\right)^\frac12\left(\inte \hat{n}^6|\psi_n|^2 + \hat{n}^4 \left|\psi'_n\right|^2\,dy\right)^\frac12 +\frac32\Phi|\hat{n}|^3\inte \left|\psi'_n\right|\left|\psi_n\right|\,dy
%\ea\end{equation}
%and
%\begin{equation}\label{5-1-8}\ba
%\frac32\Phi|\hat{n}|^3\inte \left|\psi_n'\right|\left|\psi_n\right|\,dy
%\leq&~~
%C\Phi^\frac12\left(\inte |\BF_{n}|^2\,dy\right)^\frac14\left(\inte |\hat{n}|^3 \left|\psi_n'\right|^2 \,dy\right)^\frac12\\
%&~~\cdot\left(\inte \hat{n}^6|\psi_n|^2 + \hat{n}^4 \left|\psi'_n\right|^2\,dy\right)^\frac14.
%\ea\end{equation}
%Substituting \eqref{5-1-8} into \eqref{5-1-7} and using Young's inequality yield
%\begin{equation}\label{5-1-9}
%\inte  \hat{n}^2\left|\psi_n''\right|^2+  \hat{n}^4 \left|\psi_n'\right|^2+\hat{n}^6|\psi_n|^2\,dy\leq C(1+\Phi^2)\inte |\BF_{n}|^2\,dy.
%\end{equation}
%Hence, one has
%\begin{equation}\label{5-1-10}
%\|\partial_x \Bv_n\|_{L^2(\Omega)}\leq C(1+\Phi)\|\BF_n\|_{L^2(\Omega)}.
%\end{equation}

Finally, as shown in Proposition \ref{back}, $\Bv_n$ satisfies the equation
\be \label{stokes}
\left\{  \ba
& -\Delta \Bv_n + \boldsymbol{U}\partial_x  \Bv_n + v_{2,n}e^{i\hat{n} x}  \boldsymbol{U}' + \nabla P_n = \BF_n, \ \ \ \mbox{in}\ \ \Omega, \\
& {\rm div}~\Bv_n = 0,\ \ \ \ \ \ \ \ \ \ \ \ \ \ \ \ \ \ \ \ \ \ \ \ \ \ \ \ \ \ \ \ \ \ \ \ \ \ \ \ \ \ \ \ \ \ \mbox{in} \ \ \Omega.
\ea  \right.
\ee
According to the regularity theory for Stokes equations (\cite[Lemma VI.1.2]{Galdi}), one has
\be \label{est-stokes}
 \ba
\|\Bv_n\|_{H^2 (\Omega)} & \leq C \|\BF_n\|_{L^2(\Omega)} +C\Phi \|\partial_x \Bv_n\|_{L^2(\Omega)} + C\Phi \|y v_{2,n}e^{i\hat{n} x}\|_{L^2(\Omega)} + C \|\Bv_n\|_{H^1(\Omega)}\\
& \leq C (1 + \Phi^2 ) \|\BF_n\|_{L^2(\Omega)} .
\ea
\ee
This finishes the proof of the proposition.
\end{proof}

\subsection{Estimate for the $0$-th mode}
To establish the uniform estimate when the flux $\Phi$ is large,  the $0$-th mode solution is investigated in detail in this subsection.
\begin{pro}\label{0mode}
Let $\psi_0$ be the solution obtained in Proposition \ref{existence-stream}. Then the corresponding velocity field $\Bv_0=-\psi_0'\Be_1$ satisfies
\begin{equation}\nonumber
\|\Bv_0\|_{H^2(\Omega)}\le C_1\|F_{1,0}\|_{L^2(\Omega)},
\end{equation}
where $C_1>0$ is a uniform constant independent of flux $\Phi$, $F_{1,0}$ and $L$.
\end{pro}
\begin{proof} For $n=0$, the problem \eqref{stream} and \eqref{streamBC} becomes
\be \label{5-2-1}
\left\{\begin{aligned}
&\psi_0^{(4)}=F_{1,0}',\\
&\psi_0(\pm1)=\psi_0'(\pm1)=0.
\end{aligned}\right.
\ee

The solution of the problem \eqref{5-2-1} can be written as
\begin{equation}\nonumber
\psi_{0}(y)=\int_{-1}^y\int_{-1}^\tau\int_{-1}^tF_{1,0}(s)\,dsdtd\tau+\frac{A_1}{6}(y+1)^3+\frac{A_2}{2}(y+1)^2,
\end{equation}
where
\begin{equation}\nonumber
A_1=\frac32\int_{-1}^1\int_{-1}^\tau\int_{-1}^tF_{1,0}(s)\,dsdtd\tau-\frac32\int_{-1}^1\int_{-1}^tF_{1,0}(s)\,dsdt
\end{equation}
and
\begin{equation}\nonumber
A_2=\int_{-1}^1\int_{-1}^tF_{1,0}(s)\,dsdt-\frac32\int_{-1}^1\int_{-1}^\tau\int_{-1}^tF_{1,0}(s)\,dsdtd\tau.
\end{equation}
Noting $\Bv_0=-\psi_0'(y)\Be_1$ and using H\"{o}lder inequality yield
\begin{equation}\nonumber
\|\Bv_0\|_{H^2(\Omega)}^2 \leq CL\inte \left|\psi_0^{(3)}\right|^2+|\psi_0''|^2+|\psi_0'|^2\,dy \leq CL\inte |F_{1,0}|^2\,dy\leq C\|F_{1,0}\Be_1\|_{L^2(\Omega)}^2.
\end{equation}
This finishes the proof of Proposition \ref{0mode}.
\end{proof}

\subsection{Uniform estimate for the case with large flux and intermediate frequency}\label{secinter}
In this subsection, the uniform estimate for the solution of \eqref{stream} and \eqref{streamBC} with respect to the flux $\Phi$ is established when the flux is large and the frequency is in the intermediate regime. Inspired by \cite{GHM, WX1},  the solutions are decomposed into several parts. The first part is the solution of \eqref{stream} supplemented with the Navier slip boundary conditions and the second part is the associated boundary layer. The other parts are used to recover the equation and the no-slip boundary condition.

\begin{pro}\label{mediumstream} Assume that $\Phi \gg 1$.
There exists a small constant $\epsilon_1 \in (0, 1)$, which is independent of $\Phi$, such that as long as $1\leq |n|\leq \epsilon_1 L\sqrt{\Phi} $,
the solution $\psi_n(y)$ to the problem \eqref{stream} and \eqref{streamBC} can be decomposed into five parts,
\be \label{decompose}\ba
\psi_n(y) = &\psi_{n,s}(y) +  b_n^o\left[\psi_{n,BL}^o(y) +\psi_{n,e}^o(y) \right]+b_n^e\left[\psi_{n,BL}^e(y) +\psi_{n,e}^e(y) \right] \\
& + a_n^o[\psi_{n,p}^o(y)+\psi_{n,r}^o(y)]+a_n^e[\psi_{n,p}^e(y)+\psi_{n,r}^e(y)].
\ea\ee
Here $(1)$\ $\psi_{n,s}$ is a solution to the following problem with slip boundary conditions
\be \label{slip}
\left\{\ba
&-i\hat{n}U''(y)\psi_{n,s}+ i\hat{n}U(y)\left(\frac{d^2}{d y^2}-\hat{n}^2\right)\psi_{n,s}-\left(\frac{d^2}{d y^2}-\hat{n}^2\right)^2\psi_{n,s}= i\hat{n} F_{2, n}- \frac{d}{dy} F_{1, n} ,\\
&\psi_{n,s}(\pm1)=\psi_{n,s}''(\pm1)=0.
\ea\right.\ee
Moreover, $\psi_{n,s}$ satisfies
\be \label{est5-3}
 \inte  | \psi_{n,s}' |^2
+\hat{n}^2| \psi_{n,s} |^2 \, dy \leq C |\Phi\hat{n}|^{-\frac43} \inte |\BF_n|^2 \, dy,
\ee
\be \label{est5-4}
\inte |\psi_{n,s}''|^2  +  \hat{n}^2|\psi_{n,s}' |^2  + \hat{n}^4|\psi_{n,s}|^2
 \leq  C |\Phi\hat{n}|^{-\frac23}  \inte |\BF_n|^2 \, dy,
\ee
and
\be \label{est5-5}
\inte \left|  \psi_{n,s}^{(3)}\right|^2 +
\hat{n}^2|\psi_{n,s}'' |^2  + \hat{n}^4|\psi_{n,s}'|^2 +\hat{n}^6|\psi_{n,s} |^2 \, dy
\leq  C\inte |\BF_n|^2 \, dy.
\ee

$(2)$ $\psi_{n,BL}^e$ and $\psi_{n,BL}^o$ are the boundary layer functions
\be \nonumber
\psi_{n,BL}^e=\chi^+(y)\psi_{n,BL}^+(y)+\chi^-(y)\psi_{n,BL}^-(y)
\ \
\text{and}
\ \
\psi_{n,BL}^o=\chi^+(y)\psi_{n,BL}^+(y)-\chi^-(y)\psi_{n,BL}^-(y).
\ee
Here
\be \nonumber
\psi_{n,BL}^\pm(y)=C_{0,n,\Phi}G_{n,\Phi}(\beta(1\mp y))
\ee
and $G_{n,\Phi}(\rho)$ is a smooth function, decaying exponentially at infinity, uniformly bounded in the set
\be\nonumber
\mathcal{E}=\{(n,\Phi,\rho):\Phi\ge 1,~1\le |n|\leq L_0\sqrt{\Phi},~0\le \rho<\infty\}.
\ee
$\chi^+(y)=\chi^-(-y)$ are smooth cut-off functions on $[-1,1]$ satisfying that
\be \label{5-3-5}
\chi^+(y) = \left\{ \ba  &  1, \ \ \ \ y\geq \frac12,  \\ & 0, \ \ \  \ y\leq \frac14.   \ea  \right.
\ee
$\psi_{n,p}^e$ and $\psi_{n,p}^o$ are irrotational flows defined by
\be\nonumber
\psi_{n,p}^e=e^{\hat{n}y}+e^{-\hat{n}y}
\quad
\text{and}
\quad
\psi_{n,p}^o=e^{\hat{n}y}-e^{-\hat{n}y}.
\ee
The coefficients $a_n^e$, $a_n^o$, $b_n^e$ and $b_n^o$ satisfy
\be \nonumber
|b_n^e|+|b_n^o|\leq C|\Phi\hat{n}|^{-\frac56}\left(\inte |\BF_n|^2\,dy\right)^\frac12
\ee
and
\be\nonumber
|a_n^e|+|a_n^o|\leq C|\Phi\hat{n}|^{-\frac56}e^{-|\hat{n}|}\left(\inte |\BF_n|^2\,dy\right)^\frac12.
\ee

$(3)$ $\psi_{n,BL}^e+\psi_{n,e}^e$, $\psi_{n,BL}^o+\psi_{n,e}^o$, $\psi_{n,p}^e+\psi_{n,r}^e$, and  $\psi_{n,p}^o+\psi_{n,r}^o$ satisfy the equation \eqref{stream} with $f_n=0$ supplemented the boundary conditions
\be \nonumber
\psi_{n,e}^e(\pm1)=(\psi_{n,e}^e)''(\pm1)=\psi_{n,r}^e(\pm1)=(\psi_{n,r}^e)''(\pm1)=0
\ee
and
\be \nonumber
\psi_{n,e}^o(\pm1)=(\psi_{n,e}^o)''(\pm1)=\psi_{n,r}^o(\pm1)=(\psi_{n,r}^o)''(\pm1)=0.
\ee

In conclusion, $\psi_n$ satisfies
\be \label{est5-6}
\inte | \psi_n'|^2 + \hat{n}^2|\psi_n|^2 \,dy
\leq  C |\Phi\hat{n}|^{-\frac43}\inte |\BF_n|^2 \, dy
\ee
and
\be \label{est5-7}
\inte |\psi_n^{(3)}|^2 + \hat{n}^2|\psi_n''|^2 +\hat{n}^4| \psi_n' |^2 +\hat{n}^6|\psi_n|^2 \,dy
\leq C \inte |\BF_n|^2 \, dy.
\ee
\end{pro}

The rest of this subsection devotes to the proof of Proposition \ref{mediumstream}. First, one has the following a priori estimates for the problem \eqref{slip}.

\begin{lemma}\label{slip-est}
Assume that $L \leq L_0$, and $\psi_{n,s}$ is a smooth solution to the problem \eqref{slip}. Then there exists a constant $C$ independent of $n$ and $f_n$ such that
\begin{equation}\nonumber
\inte |\psi_{n,s}|^2\,dy \le C|\Phi\hat{n}|^{-2}\inte |f_n|^2\,dy,
\end{equation}
\begin{equation}\nonumber
\inte \hat{n}^2|\psi_{n,s}|^2+|\psi_{n,s}'|^2\,dy \le C|\Phi\hat{n}|^{-\frac53}\inte |f_n|^2\,dy,
\end{equation}
\begin{equation}\nonumber
\int_{-1}^1 \hat{n}^4|\psi_{n,s}|^2+\hat{n}^2|\psi_{n,s}'|^2+|\psi_{n,s}''|^2\,d y\leq C|\Phi\hat{n}|^{-\frac43}\int_{-1}^1 |f_n|^2\,dy,
\end{equation}
and
\begin{equation}\nonumber
\inte \left|\psi_{n,s}^{(3)}\right|^2+\hat{n}^2|\psi_{n,s}''|^2+\hat{n}^4|\psi_{n,s}'|^2+\hat{n}^6|\psi_{n,s}|^2\,dy \le C|\Phi\hat{n}|^{-\frac23}\inte |f_n|^2\,dy.
\end{equation}
Moreover, $\psi_{n, s}$ satisfies the estimates \eqref{est5-3}-\eqref{est5-5} and
\begin{equation}\nonumber
 \left|\psi_{n,s}'(\pm1)\right|\le C|\Phi\hat{n}|^{-\frac12}\left(\inte |\BF_n|^2\,dy\right)^\frac12.
\end{equation}
\end{lemma}
\begin{proof}
{\em Step 1. Basic  a priori estimate.} As in the proof of Lemma \ref{basic-est1}, multiplying the equation \eqref{slip} by $\overline{\psi_{n,s}}$ and integrating the resulting equation over $[-1,1]$ yield
\begin{equation}\label{5-3-9}
\int_{-1}^1 \hat{n}^4|\psi_{n,s}|^2+2\hat{n}^2|\psi_{n,s}'|^2+|\psi''_{n,s}|^2\,d y=-\Re\int_{-1}^1 f_n\overline{\psi_{n,s}}\,dy+\Im\int_{-1}^1 \hat{n}U'\psi_{n,s}'\overline{\psi_{n,s}}\,dy
\end{equation}
and
\begin{equation}\nonumber
\frac{3\Phi\hat{n}}{4}\int_{-1}^1 \hat{n}^2|\psi_{n,s}|^2(1-y^2)+|\psi_{n,s}'|^2(1-y^2)\,d y=-\Im\int_{-1}^1 f_n\overline{\psi_{n,s}}\,dy+\frac{3\Phi\hat{n}}{4}\int_{-1}^1|\psi_{n,s}|^2\, dy.
\end{equation}
Since $L\le L_0$, one has
\begin{equation}\label{5-3-11}
|\Phi\hat{n}|\int_{-1}^1 \hat{n}^2|\psi_{n,s}|^2(1-y^2)+|\psi_{n,s}'|^2(1-y^2)\,d y\le C\left|\int_{-1}^1 f_n\overline{\psi_{n,s}}\,dy\right|.
\end{equation}
It follows from Lemma \ref{lemmaHLP} that the estimate
\begin{equation}\nonumber
|\Phi\hat{n}|\int_{-1}^1|\psi_{n,s}|^2\,d y\le C\left|\int_{-1}^1 f_n\overline{\psi_{n,s}}\,dy\right|
\end{equation}
holds. Using Cauchy-Schwarz inequality gives
\begin{equation}\label{5-3-13}
|\Phi\hat{n}|^2\int_{-1}^1|\psi_{n,s}|^2\,d y\le C\int_{-1}^1|f_n|^2\,dy.
\end{equation}

{\em Step 2. The first and second order estimates.} Note that the problem \eqref{slip} can be written as
\begin{equation}\label{rslip}
\left\{\ba
&i\frac{3\Phi\hat{n}}{4}(1-y^2)\Psi_{n,s}-\left(\frac{d^2}{d y^2}-\hat{n}^2\right)\Psi_{n,s}=f_n+i\hat{n}U''(y)\psi_{n,s}=:\tilde{f}_n,\\
&\left(\frac{d^2}{d y^2}-\hat{n}^2\right)\psi_{n,s}=\Psi_{n,s},\\
&\psi_{n,s}(\pm1)=\Psi_{n,s}(\pm1)=0.
\ea\right.
\end{equation}
Multiplying  the first equation of \eqref{rslip} by $\overline{\Psi_{n,s}}$ and integrating the resulting equation over $[-1,1]$ gives
\begin{equation}\label{5-3-14}
\int_{-1}^1|\Psi_{n,s}'|^2+\hat{n}^2|\Psi_{n,s}|^2+\frac{3|\Phi\hat{n}|}{4}(1-y^2)|\Psi_{n,s}|^2\,dy \leq \left|\int_{-1}^1\tilde{f}_n\overline{\Psi_{n,s}}\,dy\right|.
\end{equation}
On the other hand, the straightforward computations yield
\begin{equation}\label{5-3-15}
\begin{aligned}
\int_{-1}^1\left|\tilde{f}_n\right|^2\,dy=&\int_{-1}^1\left|\left(\frac{d^2}{dy^2}-\hat{n}^2\right)\Psi_{n,s}\right|^2+\left|\frac{3\Phi\hat{n}}{4}(1-y^2)\Psi_{n,s}\right|^2\,dy\\
&+\frac{3\Phi\hat{n}}{2}\Im\int_{-1}^1(1-y^2)\Psi_{n,s}\left(\frac{d^2}{dy^2}-\hat{n}^2\right)\overline{\Psi_{n,s}}\,dy.
\end{aligned}\end{equation}
Note that
\begin{equation}\nonumber
\begin{aligned}
&\Im\int_{-1}^1(1-y^2)\Psi_{n,s}\left(\frac{d^2}{dy^2}-\hat{n}^2\right)\overline{\Psi_{n,s}}\,dy\\
=&-\Im\int_{-1}^1(1-y^2)'\Psi_{n,s}\overline{\Psi_{n,s}'}\,dy=\Im\int_{-1}^12y\Psi_{n,s}\overline{\Psi_{n,s}'}\,dy.
\end{aligned}\end{equation}
Therefore, it follows that
\begin{equation}\nonumber
\begin{aligned}
&\int_{-1}^1\left|\left(\frac{d^2}{dy^2}-\hat{n}^2\right)\Psi_{n,s}\right|^2+\left|\Phi\hat{n}(1-y^2)\Psi_{n,s}\right|^2\,dy\\
\leq&C\int_{-1}^1\left|\tilde{f}_n\right|^2\,dy+C|\Phi\hat{n}|\int_{-1}^1\left|\Psi_{n,s}\overline{\Psi_{n,s}'}\right|\,dy.
\end{aligned}\end{equation}
Lemma \ref{interpolation}, together with the slip boundary conditions $\Psi_{n,s}(\pm1)=0$, yields
\begin{equation}\nonumber
\begin{aligned}
\int_{-1}^1\left|\Psi_{n,s}\overline{\Psi_{n,s}'}\right|\,dy  \leq& \left(\int_{-1}^1\left|\Psi_{n,s}\right|^2\,dy\right)^\frac12
\left(\int_{-1}^1\left|\Psi_{n,s}'\right|^2\,dy\right)^\frac12 \\
\leq & C\left(\int_{-1}^1\left|(1-y^2)\Psi_{n,s}\right|^2\,dy\right)^\frac14
\left(\int_{-1}^1\left|\Psi_{n,s}'\right|^2\,dy\right)^\frac34.
\end{aligned}
\end{equation}
Combining this with \eqref{5-3-15} gives
\begin{equation}\label{5-3-19}
\begin{aligned}
&\int_{-1}^1\left|\left(\frac{d^2}{dy^2}-\hat{n}^2\right)\Psi_{n,s}\right|^2+\left|\Phi\hat{n}(1-y^2)\Psi_{n,s}\right|^2\,dy\\
\leq&C\int_{-1}^1\left|\tilde{f}_n\right|^2\,dy+C|\Phi\hat{n}|\left(\int_{-1}^1\left|(1-y^2)\Psi_{n,s}\right|^2\,dy\right)^\frac14
\left(\int_{-1}^1\left|\Psi_{n,s}'\right|^2\,dy\right)^\frac34\\
\leq &\frac12 \inte \left|\Phi\hat{n}(1-y^2)\Psi_{n,s}\right|^2\,dy+C|\Phi\hat{n}|^\frac23\int_{-1}^1\left|\Psi_{n,s}'\right|^2\,dy+C\int_{-1}^1\left|\tilde{f}_n\right|^2\,dy\\
\leq &\frac12 \inte \left|\Phi\hat{n}(1-y^2)\Psi_{n,s}\right|^2\,dy+C|\Phi\hat{n}|^\frac43\int_{-1}^1\left|\Psi_{n,s}\right|^2\,dy+C\int_{-1}^1\left|\tilde{f}_n\right|^2\,dy,
\end{aligned}
\end{equation}
where the estimate \eqref{5-3-14} has been used to get the last inequality.

Furthermore, the estimates \eqref{5-3-14} and \eqref{5-3-19}, together with Lemma \ref{interpolation}, yield
\begin{equation}\nonumber
\begin{aligned}
&\int_{-1}^1|\Psi_{n,s}|^2\,dy   \leq C\left(\int_{-1}^1\left|(1-y^2)\Psi_{n,s}\right|^2\,dy\right)^\frac12\left(\int_{-1}^1\left|\Psi_{n,s}'\right|^2\,dy\right)^\frac12 \\
\leq &C|\Phi\hat{n}|^{-1}\left(\int_{-1}^1 |\Phi\hat{n}|^\frac43 \left|\Psi_{n,s}\right|^2 + \left|\tilde{f}_n\right|^2\,dy
\right)^\frac12 \cdot\left(\int_{-1}^1\left|\Psi_{n,s}\right|^2\,dy\right)^\frac14\left(\int_{-1}^1\left|\tilde{f}_n\right|^2\,dy\right)^\frac14.
\end{aligned}
\end{equation}
Applying Young's inequality gives
\begin{equation}\label{5-3-23}
\int_{-1}^1|\Psi_{n,s}|^2\,dy \leq C|\Phi\hat{n}|^{-\frac43}\int_{-1}^1\left|\tilde{f}_n\right|^2\,dy.
\end{equation}
It follows from \eqref{5-3-13} that one has
\begin{equation}\label{5-3-25}
\int_{-1}^1\left|\tilde{f}_n\right|^2\,dy \leq \int_{-1}^1\left|f_n\right|^2\,dy+C|\Phi \hat{n}|^2\int_{-1}^1\left|\psi_{n,s}\right|^2\,dy\leq C\int_{-1}^1\left|f_n\right|^2\,dy.
\end{equation}
Combining \eqref{5-3-19}, \eqref{5-3-23} and \eqref{5-3-25} together yields
\begin{equation}\nonumber
\begin{aligned}
&\int_{-1}^1\left|\left(\frac{d^2}{dy^2}-\hat{n}^2\right)\Psi_{n,s}\right|^2 +|\Phi\hat{n}|^2 \left|(1-y^2)\Psi_{n,s}\right|^2 +|\Phi\hat{n}|^\frac43 \left|\Psi_{n,s}\right|^2\,dy
\leq C\int_{-1}^1\left|f_n\right|^2\,dy.
\end{aligned}
\end{equation}

Note that
\begin{equation}\label{5-3-27}\begin{aligned}
\int_{-1}^1\left|\Psi_{n,s}\right|^2\,dy=&\int_{-1}^1\left|\psi_{n,s}''-\hat{n}^2\psi_{n,s}\right|^2\,dy\\
=&\int_{-1}^1\left|\psi_{n,s}''\right|^2-\hat{n}^2\psi_{n,s}''\overline{\psi_{n,s}}-\hat{n}^2\overline{\psi_{n,s}''}\psi_{n,s}+\hat{n}^4\left|\psi_{n,s}\right|^2\,dy\\
=&\int_{-1}^1\left|\psi_{n,s}''\right|^2+2\hat{n}^2|\psi_{n,s}'|^2+\hat{n}^4\left|\psi_{n,s}\right|^2\,dy.
\end{aligned}\end{equation}
Thus one has
\begin{equation}\label{5-3-28}\begin{aligned}
\int_{-1}^1\left|\psi_{n,s}''\right|^2+\hat{n}^2|\psi_{n,s}'|^2+\hat{n}^4\left|\psi_{n,s}\right|^2\,dy\leq C|\Phi\hat{n}|^{-\frac43}\int_{-1}^1 |f_n|^2\,dy.
\end{aligned}\end{equation}

It follows from \eqref{5-3-13}, \eqref{5-3-28} and Lemma \ref{lemmaA1} that
\begin{equation}\nonumber
\ba
&~~\inte |\psi_{n,s}'|^2+\hat{n}^2|\psi_{n,s}|^2\,dy\\
\leq&~~C\left(\inte |\psi_{n,s}|^2\,dy\right)^\frac12\left[ \left(\inte |\psi_{n,s}''|^2\,dy\right)^\frac12+\left(\inte \hat{n}^4|\psi_{n,s}|^2\,dy\right)^\frac12\right]\\
\leq&~~C|\Phi\hat{n}|^{-\frac53}\inte |f_n|^2\,dy.
\ea\end{equation}

{\em Step 3. Higher order estimate.} Note that
\begin{equation}\label{5-3-30}
\int_{-1}^1\left|\Psi_{n,s}''\right|^2+2\hat{n}^2\left|\Psi_{n,s}'\right|^2+\hat{n}^4\left|\Psi_{n,s}\right|^2\,dy= \int_{-1}^1\left|\left(\frac{d^2}{dy^2}-\hat{n}^2\right)\Psi_{n,s}\right|^2\,dy \leq C\int_{-1}^1\left|f_n\right|^2\,dy.
\end{equation}
Similar to \eqref{5-3-27}, one also has
\begin{equation}\nonumber
\int_{-1}^1\left|\Psi_{n,s}''\right|^2\,dy =\int_{-1}^1\left|\psi_{n,s}^{(4)}\right|^2+2\hat{n}^2\left|\psi_{n,s}^{(3)}\right|^2+\hat{n}^4\left|\psi_{n,s}''\right|^2\,dy
\end{equation}
and
\begin{equation}\nonumber
\int_{-1}^1\left|\Psi_{n,s}'\right|^2\,dy =\int_{-1}^1\left|\psi_{n,s}^{(3)}\right|^2+2\hat{n}^2\left|\psi_{n,s}''\right|^2+\hat{n}^4\left|\psi_{n,s}'\right|^2\,dy.
\end{equation}
Hence it follows from \eqref{5-3-30} that
\begin{equation}\nonumber
\int_{-1}^1\left|\psi_{n,s}^{(4)}\right|^2+\hat{n}^2\left|\psi_{n,s}^{(3)}\right|^2+\hat{n}^4\left|\psi_{n,s}''\right|^2
+\hat{n}^6\left|\psi_{n,s}'\right|^2+\hat{n}^8\left|\psi_{n,s}\right|^2\,dy
\leq C\int_{-1}^1\left|f_n\right|^2\,dy.
\end{equation}
This, together with \eqref{5-3-28}, yields
\begin{equation}\nonumber
\int_{-1}^1\left|\psi_{n,s}^{(3)}\right|^2+\hat{n}^2\left|\psi_{n,s}''\right|^2+\hat{n}^4\left|\psi_{n,s}'\right|^2
+\hat{n}^6\left|\psi_{n,s}\right|^2\,dy
\leq C|\Phi\hat{n}|^{-\frac23}\int_{-1}^1\left|f_n\right|^2\,dy.
\end{equation}

{\em Step 4. The proof of a priori estimate \eqref{est5-3}-\eqref{est5-5}.} Multiplying \eqref{slip} by $\overline{\psi_{n,s}''}$ and integrating the resulting equation over $[-1,1]$ yield
\begin{equation}\label{5-3-36}
\int_{-1}^1 \overline{\psi_{n,s}''}\left[-i\hat{n}  U''+ i\hat{n} U\left(\frac{d^2}{d y^2}-\hat{n}^2\right)-\left(\frac{d^2}{d y^2}-\hat{n}^2\right)^2\right]\psi_{n,s} \,d y=\int_{-1}^1 \overline{\psi_{n,s}''}f_n\,dy.
\end{equation}
It follows from integration by parts and the homogeneous boundary conditions that
\begin{equation}\nonumber
\int_{-1}^1-i\hat{n} U''\psi_{n,s}\overline{\psi_{n,s}''}\,dy=\int_{-1}^1i\hat{n} U''|\psi_{n,s}'|^2\,dy
\end{equation}
and
\begin{equation}\nonumber
\int_{-1}^1 i\hat{n} U \left(\frac{d^2}{d y^2}-\hat{n}^2\right)\psi_{n,s} \overline{\psi_{n,s}''}\,dy=\int_{-1}^1i\hat{n} U|\psi_{n,s}''|^2 +i\hat{n}^3 U|\psi_{n,s}'|^2+i\hat{n}^3 U'\psi_{n,s}\overline{\psi_{n,s}'}\,dy.
\end{equation}
Furthermore, one has
\begin{equation}\nonumber
\begin{aligned}
\int_{-1}^1-\left(\frac{d^2}{d y^2}-\hat{n}^2\right)^2\psi_{n,s} \overline{\psi_{n,s}''}\,dy
=&\int_{-1}^1-\hat{n}^4\psi_{n,s} \overline{\psi_{n,s}''} +2\hat{n}^2|\psi_{n,s}''|^2-\psi_{n,s}^{(4)} \overline{\psi_{n,s}''}\,dy\\
=&\int_{-1}^1\hat{n}^4|\psi_{n,s}'|^2+
2\hat{n}^2|\psi_{n,s}''|^2+\left|\psi^{(3)}_n\right|^2\,dy.
\end{aligned}
\end{equation}
Then one can decompose \eqref{5-3-36} into its real and imaginary parts as
\begin{equation}\label{5-3-40}
\int_{-1}^1 \hat{n}^4|\psi_{n,s}'|^2+2\hat{n}^2|\psi_{n,s}''|^2+\left|\psi_{n,s}^{(3)}\right|^2\,d y=\Re\int_{-1}^1 f_n\overline{\psi_{n,s}''}\,dy+\Im\int_{-1}^1 \hat{n}^3 U'\psi_{n,s}\overline{\psi_{n,s}'}\,dy
\end{equation}
and
\begin{equation}\label{5-3-41}
\int_{-1}^1 \hat{n}U''|\psi_{n,s}'|^2+ \hat{n}U|\psi_{n,s}''|^2+ \hat{n}^3 U|\psi_{n,s}'|^2\,d y+\Re\int_{-1}^1 \hat{n}^3 U'\psi_{n,s}\overline{\psi_{n,s}'}\,dy=\Im\int_{-1}^1 f_n\overline{\psi_{n,s}''}\,dy,
\end{equation}
respectively. Note that
\begin{equation}\nonumber
\Re\int_{-1}^1 \hat{n}^3 U'\psi_{n,s}\overline{\psi_{n,s}'}\,dy=-\frac12\int_{-1}^1 \hat{n}^3 U''|\psi_{n,s}|^2\,dy.
\end{equation}
The equation \eqref{5-3-41} can be rewritten as follows,
\begin{equation}\label{5-3-43}\begin{aligned}
&\frac{3\Phi\hat{n}}{4}\int_{-1}^1 \hat{n}^2|\psi_{n,s}'|^2(1-y^2)+|\psi_{n,s}''|^2(1-y^2)+ \hat{n}^2 |\psi_{n,s}|^2\,d y\\
=&-\Im\int_{-1}^1 f_n\overline{\psi_{n,s}''}\,dy+\frac{3\Phi\hat{n}}{2}\int_{-1}^1 |\psi_{n,s}'|^2\, dy.
\end{aligned}\end{equation}
On the other hand, multiplying \eqref{5-3-9} and \eqref{5-3-11} by $\hat{n}^2$ yields
\begin{equation}\label{5-3-44}
\int_{-1}^1  \hat{n}^6|\psi_{n,s}|^2+2\hat{n}^4|\psi_{n,s}'|^2+\hat{n}^2|\psi_{n,s}''|^2\,d y=-\Re\int_{-1}^1 \hat{n}^2f_n\overline{\psi_{n,s}}\,dy+\Im\int_{-1}^1 \hat{n}^3 U'\psi_{n,s}'\overline{\psi_{n,s}}\,dy
\end{equation}
and
\begin{equation}\label{5-3-45}
|\Phi\hat{n}|\int_{-1}^1\hat{n}^4|\psi_{n,s}|^2(1-y^2)+ \hat{n}^2|\psi_{n,s}'|^2(1-y^2)\,d y\leq C\left|\int_{-1}^1 \hat{n}^2f_n\overline{\psi_{n,s}}\,dy\right|.
\end{equation}
Summing \eqref{5-3-40} and \eqref{5-3-44} yields
\begin{equation}\label{5-3-46}
\int_{-1}^1 \hat{n}^6|\psi_{n,s}|^2+3\hat{n}^4|\psi_{n,s}'|^2+3\hat{n}^2|\psi''_n|^2+\left|\psi_{n,s}^{(3)}\right|^2\,d y=\Re\int_{-1}^1 f_n\overline{\psi_{n,s}''}\,dy-\Re\int_{-1}^1 \hat{n}^2 f_n\overline{\psi_{n,s}}\,dy.
\end{equation}
By virtue of integration by parts and Cauchy-Schwarz inequality, one has
\begin{equation}\label{5-3-47}\ba
&~~\left|\int_{-1}^1\Re\int_{-1}^1 f_n\overline{\psi_{n,s}''}\,dy\right|=\left|\Re\int_{-1}^1 i\hat{n} F_{2,n}\overline{\psi_{n,s}''}+F_{1,n}\overline{\psi_{n,s}^{(3)}}\,dy\right|\\
\leq&~~ C\inte |\BF_n|^2\,dy+\frac12\inte \hat{n}^2|\psi_{n,s}''|^2\,dy+\frac12\inte|\psi_{n,s}^{(3)}|^2\,dy
\ea\end{equation}
and
\begin{equation}\label{5-3-48}\ba
&~~\left|\Re\int_{-1}^1 \hat{n}^2f_n\overline{\psi_{n,s}}\,dy\right|=\left|\Re\int_{-1}^1 i\hat{n}^3 F_{2,n}\overline{\psi_{n,s}}+\hat{n}^2F_{1,n}\overline{\psi_{n,s}'}\,dy\right|\\
\leq&~~ C\inte |\BF_n|^2\,dy+\frac12\inte \hat{n}^6|\psi_{n,s}|^2\,dy+\frac12\inte \hat{n}^4|\psi_{n,s}'|^2\,dy.
\ea\end{equation}
Combining \eqref{5-3-46}-\eqref{5-3-48} gives \eqref{est5-5}. Furthermore, it follows from \eqref{est5-5} and \eqref{5-3-43} that
\begin{equation}\nonumber
\ba
&~~\frac{3 | \Phi\hat{n}| }{4}\int_{-1}^1 \hat{n}^2|\psi_{n,s}'|^2(1-y^2)+  |\psi_{n,s}''|^2(1-y^2)+ \hat{n}^2|\psi_{n,s}|^2\,d y\\
\leq &~~\left|\int_{-1}^1 i\hat{n} F_{2,n}\overline{\psi_{n,s}''}+F_{1,n}\overline{\psi_{n,s}^{(3)}}\,dy\right|+\frac{3|\Phi\hat{n}|}{2}\int_{-1}^1 |\psi_{n,s}'|^2\, dy\\
\leq &~~C\left(\int_{-1}^1|\BF_n|^2\,dy\right)^\frac12\left(\int_{-1}^1\hat{n}^2|\psi_{n,s}''|^2+\left|\psi_{n,s}^{(3)}\right|^2\,dy\right)^\frac12+\frac{3|\Phi\hat{n}|}{2}\int_{-1}^1 |\psi_{n,s}'|^2\, dy\\
\leq &~~C\int_{-1}^1|\BF_n|^2\,dy+\frac{3|\Phi\hat{n}|}{2}\int_{-1}^1 |\psi_{n,s}'|^2\, dy.
\ea\end{equation}
It follows from Lemma \ref{lemmaHLP} that one has
\begin{equation}\label{5-3-50}
|\Phi\hat{n}|\int_{-1}^1\hat{n}^2(1-y^2)|\psi_{n,s}'|^2+(1-y^2)|\psi_{n,s}''|^2+\hat{n}^2|\psi_{n,s}|^2\,d y
\leq C\int_{-1}^1|\BF_n|^2\,dy,
\end{equation}
provided $L\leq L_0$.
Similarly, using \eqref{5-3-45} and \eqref{est5-5} yields
\begin{equation}\label{5-3-51}\ba
&~~|\Phi\hat{n}|\int_{-1}^1\hat{n}^4|\psi_{n,s}|^2(1-y^2)
+\hat{n}^2|\psi_{n,s}'|^2(1-y^2)\,d y\\
\leq &~~C\left(\int_{-1}^1|\BF_n|^2\,dy\right)^\frac12
\left(\int_{-1}^1 \hat{n}^6|\psi_{n,s}|^2+\hat{n}^4\left|\psi_{n,s}'\right|^2\,dy\right)^\frac12\\
\leq &~~C\int_{-1}^1|\BF_n|^2\,dy.
\ea\end{equation}
With the aid of Lemma \ref{weightinequality}, combining \eqref{est5-5} and \eqref{5-3-50} gives
\be\label{5-3-52}\ba
\inte |\psi_{n,s}''|^2\,dy\leq &~~C\left(\inte|\psi_{n,s}''|^2(1-y^2)\,dy\right)^\frac23\left(\inte\left|\psi_{n,s}^{(3)}\right|^2\,dy\right)^\frac13\\
&~~+C\inte |\psi_{n,s}''|^2(1-y^2)\,dy\\
\leq&~~C|\Phi\hat{n}|^{-\frac23}\int_{-1}^1|\BF_n|^2\,dy.
\ea\ee
Similarly, it follows from Lemma \ref{weightinequality}, \eqref{est5-5} and \eqref{5-3-51} that one has
\be\label{5-3-53}\ba
\inte \hat{n}^4|\psi_{n,s}|^2\,dy\leq &~~C\left(\inte \hat{n}^4|\psi_{n,s}|^2(1-y^2)\,dy\right)^\frac23\left(\inte \hat{n}^4|\psi_{n,s}'|^2\,dy\right)^\frac13\\
&~~+C\inte \hat{n}^4|\psi_{n,s}|^2(1-y^2)\,dy\\
\leq&~~C|\Phi\hat{n}|^{-\frac23}\int_{-1}^1|\BF_n|^2\,dy.
\ea\ee
The estimates  \eqref{5-3-52}-\eqref{5-3-53}, together with Lemma \ref{lemmaA1}, give
\be\label{5-3-54}\ba
\inte \hat{n}^2|\psi_{n,s}'|^2\,dy\leq
C\left(\inte \hat{n}^4|\psi_{n,s}|^2\,dy\right)^\frac12\left(\inte|\psi_{n,s}''|^2\,dy\right)^\frac12
\leq C|\Phi\hat{n}|^{-\frac23}\int_{-1}^1|\BF_n|^2\,dy.
\ea\ee

{\em Step 5. Weighted estimates and boundary estimate of $\psi_{n,s}$}. Applying integration by parts to the inequality \eqref{5-3-11} yields
\begin{equation}\label{5-3-55}
\ba
&~~\int_{-1}^1\hat{n}^2|\psi_{n,s}|^2(1-y^2)
+|\psi_{n,s}'|^2(1-y^2)\,d y\\
\leq &~~C|\Phi\hat{n}|^{-1}\left(\int_{-1}^1|\BF_n|^2\,dy\right)^\frac12
\left(\int_{-1}^1\hat{n}^2|\psi_{n,s}|^2+\left|\psi_{n,s}'\right|^2\,dy\right)^\frac12.
\ea\end{equation}
It follows from Lemma \ref{weightinequality} and \eqref{5-3-52}-\eqref{5-3-55} that one has
\begin{equation*}\ba
&~~\int_{-1}^1\hat{n}^2|\psi_{n,s}|^2+\left|\psi_{n,s}'\right|^2\,dy\\
\leq&~~C\left(\inte \hat{n}^2|\psi_{n,s}|^2(1-y^2)\,dy\right)^\frac23\left(\inte \hat{n}^2|\psi_{n,s}'|^2\,dy\right)^\frac13+C\inte \hat{n}^2|\psi_{n,s}|^2(1-y^2)\,dy\\
&~~+C\left(\inte |\psi_{n,s}'|^2(1-y^2)\,dy\right)^\frac23\left(\inte |\psi_{n,s}''|^2\,dy\right)^\frac13+C\inte|\psi_{n,s}'|^2(1-y^2)\,dy\\
\leq&~~C|\Phi\hat{n}|^{-\frac29}\left(\inte \hat{n}^2|\psi_{n,s}|^2(1-y^2)\,dy\right)^\frac23\left(\inte |\BF_n|^2\,dy\right)^\frac13\\
&~~+C|\Phi\hat{n}|^{-\frac29}\left(\inte |\psi_{n,s}'|^2(1-y^2)\,dy\right)^\frac23\left(\inte |\BF_n|^2\,dy\right)^\frac13\\
&~~+C\inte \hat{n}^2|\psi_{n,s}|^2(1-y^2)\,dy+C\inte|\psi_{n,s}'|^2(1-y^2)\,dy\\
\leq&~~C|\Phi\hat{n}|^{-\frac29}|\Phi \hat{n}|^{-\frac23}\left(\int_{-1}^1\hat{n}^2|\psi_{n,s}|^2+\left|\psi_{n,s}'\right|^2\,dy\right)^\frac13\left(\inte |\BF_n|^2\,dy\right)^\frac23\\
&~~+C|\Phi\hat{n}|^{-1}\left(\int_{-1}^1|\BF_n|^2\,dy\right)^\frac12
\left(\int_{-1}^1\hat{n}^2|\psi_{n,s}|^2+\left|\psi_{n,s}'\right|^2\,dy\right)^\frac12.
\ea\end{equation*}
Using Young's inequality yields
\begin{equation}\nonumber
\inte \hat{n}^2|\psi_{n,s}|+|\psi_{n,s}'|^2 \,dy \leq C|\Phi\hat{n}|^{-\frac43}\inte |\BF_n|^2\,dy.
\end{equation}
This, together with \eqref{5-3-55}, gives
\be\nonumber
\int_{-1}^1\hat{n}^2|\psi_{n,s}|^2(1-y^2)
+|\psi_{n,s}'|^2(1-y^2)\,d y\\
\leq C|\Phi\hat{n}|^{-\frac53}\int_{-1}^1|\BF_n|^2\,dy.
\ee

We also give the estimate of $|\psi_{n,s}'(1)|$ in terms of $\BF_n$, which plays an important role in recovering the boundary value $\psi_{n}$. By Lemma \ref{lemmaA2}, it holds that
\be\nonumber
\ba
|\psi_{n,s}'(\pm1)|\leq C\left(\inte |\psi_{n,s}'|^2\,dy\right)^\frac14\left(\inte |\psi_{n,s}''|^2\,dy\right)^\frac14
\leq  C|\Phi\hat{n}|^{-\frac12}\left(\inte |\BF_n|^2\,dy\right)^\frac12.
\ea\ee
This finishes the proof of Lemma \ref{slip-est}.
\end{proof}

Similar to what has been done  in Section \ref{sec-ex}, the a priori estimates established in Lemma \ref{slip-est} together with the Galerkin method give the existence of the solutions for the problem \eqref{slip}.

Now we give the proof of Proposition \ref{mediumstream}.

\begin{proof}[Proof of Proposition \ref{mediumstream}]
For a general function $f_n$, it can be written as
\be \nonumber
f_n(y)=f_{n}^e(y)+f_{n}^o(y):=\frac{f_n(y)+f_n(-y)}{2}+\frac{f_n(y)-f_n(-y)}{2},
\ee
where $f_{n}^e(y)$ and $f_{n}^o(y)$ are even and odd functions, respectively.  Because the problem \eqref{stream}-\eqref{streamBC} is a linear problem, the solution $\psi_n$ can be written as the summation of two solutions associated with $f_n^e$ and $f_n^o$. First, we assume that $f_n$ is an even function.

{\em Step 1. Boundary layer analysis.} In order to recover the no slip boundary condition, we analyze the associated boundary layer carefully. Define the operators
\begin{equation}\nonumber
\mcA_n=i\frac{3\Phi\hat{n}}{4}(1-y^2)-\left(\frac{d^2}{dy^2}-\hat{n}^2\right)\text{ and }
~~\mcH_n=\frac{d^2}{dy^2}-\hat{n}^2.
\end{equation}
Let
\begin{equation}\nonumber
\mcA_n^\pm=i\frac{3\Phi\hat{n}}{2}(1\mp y)-\left(\frac{d^2}{dy^2}-\hat{n}^2\right).
\end{equation}
They can be regarded as the leading parts of the operator $\mcA_n$ near the boundary $y=\pm1$, respectively.

We look for two  boundary layer functions $\psi_{n,BL}^\pm(y)$, which are the solutions to
\begin{equation}\label{blayereq}
\mcA_n^\pm\mcH_n\psi_{n,BL}^\pm=0,
\end{equation}
respectively.
First, let $Ai(z)$ denote the standard Airy function which satisfies
\begin{equation}\nonumber
\frac{d^2Ai(z)}{d z^2}-zAi(z)=0, ~~\text{ in }\mathbb{C}.
\end{equation}
Define
\begin{equation}\nonumber
\widetilde{G}_{n,\Phi}(\rho)=\left\{\begin{aligned}
& Ai\left(C_+\left(\rho+\frac{2\beta \hat{n}}{3i\Phi}\right)\right),~~\text{if }n> 0,\\
&Ai\left(C_-\left(\rho+\frac{2\beta \hat{n}}{3i\Phi}\right)\right),~~\text{if }n<0,
\end{aligned}\right.
\end{equation}
where $\beta=\left|\frac{3\Phi\hat{n}}{2}\right|^\frac13$ and $C_\pm=e^{\pm i\frac{\pi}{6}}$. It's easy to check that
\begin{equation}\nonumber
\mcA_n^\pm\widetilde{G}_{n,\Phi}(\beta(1\mp y))=0.
\end{equation}
Next, set
\begin{equation}\label{5-3-61}
G_{n,\Phi}(\rho)=\int_\rho^{+\infty}e^{-\frac{|\hat{n}|}{\beta}(\rho-\tau)}
\int_\tau^{+\infty}e^{-\frac{|\hat{n}|}{\beta}(\sigma-\tau)}\widetilde{G}_{n,\Phi}(\sigma)\,d\sigma\,d\tau.
\end{equation}
The straightforward computations show that
\begin{equation}\nonumber
\frac{d^2}{d \rho^2}G_{n,\Phi}(\rho)-\frac{\hat{n}^2}{\beta^2}G_{n,\Phi}(\rho)=\widetilde{G}_{n,\Phi}(\rho).
\end{equation}
Define
\begin{equation}\nonumber
\psi_{n,BL}^\pm=C_{0,n,\Phi}G_{n,\Phi}(\beta(1\mp y))
\end{equation}
with
\begin{equation}\label{5-3-62}
C_{0,n,\Phi}=
\left\{\begin{aligned}
&\frac{1}{G_{n,\Phi}(0)},\quad &\text{if }|G_{n,\Phi}(0)|\ge 1,\\
&1,\quad &\text{otherwise}.
\end{aligned}\right.
\end{equation}
The straightforward calculations show that $|\psi^{\pm}_{n,BL}(1)|\le 1$ and satisfy \eqref{blayereq} for $y\in (-1,1)$.

Let $\chi ^+\in C^\infty([-1,1])$ be an increasing function satisfying \eqref{5-3-5} and $\chi^-(y)=\chi^+(-y)$. Define
\begin{equation}\nonumber
\psi_{n,BL}=\chi^+\psi_{n,BL}^++\chi^-\psi_{n,BL}^-.
\end{equation}
{\em Step 2. The remainder term $\psi_{n,e}$.} Suppose that $\psi_{n,e}$ satisfies the following problem
\begin{equation}\nonumber
\left\{\begin{aligned}
&-i\hat{n} U''(y)\psi_{n,e}+i\hat{n}U(y)\left(\frac{d^2}{dy^2}
-\hat{n}^2\right)\psi_{n,e}-\left(\frac{d^2}{dy^2}-\hat{n}^2\right)^2
\psi_{n,e}=Q_{n,BL},\\
&\psi_{n,e}(\pm1)=\psi_{n,e}''(\pm1)=0,
\end{aligned}\right.
\end{equation}
where
\begin{equation}\nonumber Q_{n,BL}=i\hat{n} U''(y)\psi_{n,BL}-\mcA_n\mcH_n\psi_{n,BL}.
\end{equation}

According to Lemma \ref{airy-est}, for sufficiently large $\Phi$ such that $\frac12 \beta= \frac12\left|\frac{3\Phi\hat{n}}{2}\right|^\frac13\ge R$, one has
\be\nonumber
\ba
&~~\inte  \left|i\hat{n} U''(y)\psi_{n,BL}\right|^2\,dy\leq C|\Phi\hat{n}|^2\int_\frac14^1 \left|\psi_{n,BL}^+\right|^2\,dy\\
\leq&~~C\left(\int_0^R+\int_R^{\frac34\beta}\right)\beta^5|G_{n,\Phi}(\rho)|^2\,d\rho\\
\leq&~~C\int_0^R\beta^5\,d\rho+C \int_R^{\frac34\beta}\beta^5e^{-2\rho}\,d\rho\\
\leq&~~C\beta^5,
\ea\ee
where $R$ is the constant appeared in Lemma \ref{airy-est}. Using the fact  $\mcA_n^+\mcH_n\psi_{n,BL}^+=0$ yields
\be\nonumber
\ba
-\mcA_n \mcH_n (\chi^+\psi_{n,BL}^+)=&~~
\mcA_n^+\mcH_n((1-\chi^+)\psi_{n,BL}^+)+(\mcA_n^+\mcH_n-\mcA_n\mcH_n)(\chi^+\psi_{n,BL}^+)\\
=&~~\left(i\frac{3\Phi\hat{n}}{2}(1-y)-\frac{d^2}{dy^2}+\hat{n}^2\right)\left(\frac{d^2}{dy^2}-\hat{n}^2\right)((1-\chi^+)\psi_{n,BL}^+)\\
&~~+i\frac{3\Phi\hat{n}}{4}(1-y)^2\left(\frac{d^2}{dy^2}-\hat{n}^2\right)(\chi^+\psi_{n,BL}^+).
\ea\ee
Noting that $\chi^+\psi_{n,BL}^+$ vanishes on $[-1,0]$, it suffices to estimate $\mcA_n\mcH_n(\chi^+\psi_{n,BL}^+)$ on $[0,1]$. According to Lemma \ref{airy-est}, for sufficiently large $\Phi$, one has
\be\nonumber
\ba
&~~
\int_0^1\left|\frac{3\Phi\hat{n}}{2 }(1-y)\left(\frac{d^2}{dy^2}-\hat{n}^2\right)((1-\chi^+)\psi_{n,BL}^+)\right|^2\,dy\\
\le&~~C
\int_0^\frac12\beta^6\left( |(\psi_{n,BL}^+) ''|^2+| (\psi_{n,BL}^+)'|^2
+|\psi_{n,BL}^+ |^2+\hat{n}^4|\psi_{n,BL}^+ |^2\right)\,dy\\
\le&~~C
\int_{\frac12\beta}^{\beta}\beta^9|G_{n,\Phi}''(\rho)|^2+\beta^7|G_{n,\Phi}'(\rho)|^2
+\beta^5|G_{n, \Phi }(\rho)|^2+\beta^9|G_{n,\Phi}(\rho)|^2\,d\rho\\
\le&~~C
\int_{\frac12\beta}^{\beta}(\beta^9+\beta^7
+\beta^5)e^{-2\rho}\,d\rho\\
\le&~~C e^{-\beta}(\beta^{10}+\beta^8
+\beta^6)\le C.
\ea\ee
On the other hand,
\be\nonumber
\ba
&~~\int_0^1\left|(\mcA^+\mcH_n-\mcA\mcH_n)(\chi^+\psi_{n,BL}^+)\right|^2\,dy\\
=&~~
\int_\frac14^1\left|\frac{3\Phi\hat{n}}{4}(1-y)^2\left(\frac{d^2}{dy^2}-\hat{n}^2\right)(\chi^+\psi_{n,BL}^+)\right|^2\,dy\\
\le&~~C
\int_\frac14^1\beta^6(1-y)^4\left(|\psi_{n,BL}''|^2+|(\chi^+)''|^2|\psi_{n,BL}'|^2
+|(\chi^+)'|^2|\psi_{n,BL}|^2+\hat{n}^4|\psi_{n,BL}|^2\right)\,dy\\
\le&~~C
\int_0^{\frac34\beta}\rho^4\left(\beta^5|G_{n,\Phi}''|^2+\beta^5|G_{n,\Phi}|^2\right)\,d\rho+C
\int_{\frac12\beta}^{\frac34\beta}\rho^4\left(\beta^3|G_{n,\Phi}'|^2
+\beta|G_{n,\Phi}|^2\right)\,d\rho\\
\le&~~C\beta^5\int_0^R\rho^4\,d\rho
+C\beta^5\int_R^{\frac34\beta}\rho^4e^{-2\rho}\,d\rho+C\left(\beta^3
+\beta\right)
\int_{\frac12\beta}^{\frac34\beta}\rho^4e^{-2\rho}\,d\rho\\
\leq&~~C\beta^5.
\ea\ee
Hence it follows from Lemma \ref{slip-est} that
\begin{equation}\label{5-3-73}
\inte |\psi_{n,e}|^2\,dy \le C|\Phi\hat{n}|^{-2}\beta^5\le C|\Phi\hat{n}|^{-\frac13},
\end{equation}
\begin{equation}\label{5-3-74}
\inte \hat{n}^2|\psi_{n,e}|^2+|\psi_{n,e}'|^2\,dy \le C|\Phi\hat{n}|^{-\frac53}\beta^5\le C,
\end{equation}
\begin{equation}\label{5-3-75}
\int_{-1}^1 \hat{n}^4|\psi_{n,e}|^2+\hat{n}^2|\psi_{n,e}'|^2+|\psi_{n,e}''|^2\,d y\leq C|\Phi\hat{n}|^{-\frac43}\beta^5\le C|\Phi\hat{n}|^\frac13,
\end{equation}
and
\begin{equation}\label{5-3-76}
\inte \left|\psi_{n,e}^{(3)}\right|^2+\hat{n}^2|\psi_{n,e}''|^2+\hat{n}^4|\psi_{n,e}'|^2+\hat{n}^6|\psi_{n,e}|^2\,dy \leq C|\Phi\hat{n}|^{-\frac23}\beta^5\le C|\Phi\hat{n}|.
\end{equation}

{\em Step 3. The irrotational flow $\psi_{n,p}$ and the associated error term $\psi_{n,r}$.} Denote
\begin{equation}\nonumber
\psi_{ n ,p}=e^{\hat{n}y}+e^{-\hat{n}y}.
\end{equation}
Let $\psi_{n,r}$ be the solution to the following problem,
\begin{equation}\nonumber
\left\{\begin{aligned}
&-i\hat{n} U''(y)\psi_{n,r}+i\hat{n}U(y)\left(\frac{d^2}{dy^2}
-\hat{n}^2\right)\psi_{n,r}-\left(\frac{d^2}{dy^2}-\hat{n}^2\right)^2
\psi_{n,r}=i\hat{n} U''(y)\psi_{n,p},\\
&\psi_{n,r}(\pm1)=\psi_{n,r}''(\pm1)=0.
\end{aligned}\right.
\end{equation}

Suppose that
\begin{equation}\label{recoverBC1}
\left\{\begin{aligned}
&b_n\psi_{n,BL}^+(1)+a_n\psi_{n,p}(1)=0,\\
&\psi_{n,s}'(1)+b_n[(\psi_{n,BL}^+)'(1)+\psi_{n,e}'(1)]+a_n[\psi_{n,p}'(1)+\psi_{n,r}'(1)]=0.
\end{aligned}\right.
\end{equation}
This implies that $\psi_n =\psi_{n,s}+b_n (\psi_{n,BL}+\psi_{n,e})+a_n(\psi_{n, p}+\psi_{n,r})$ satisfies the no slip boundary condition at $y=1$.
Furthermore, if $f_n$ is even with respect to $y$, the associate solution $\psi_{n,s}$ is also even. Similarly, it's easy to verify that all the components $\psi_{n,BL}$, $\psi_{n,e}$, $\psi_{n,p}$, and $\psi_{n,r}$ are all even. Thus we also have
\begin{equation}\label{recoverBC-1}
\left\{\begin{aligned}
&b_n\psi_{n,BL}^-(-1)+a_n\psi_{n,p}(-1)=0,\\
&\psi_{n,s}'(-1)+b_n[(\psi_{n,BL}^-)'(-1)+\psi_{n,e}'(-1)]+a_n[\psi_{n,p}'(-1)+\psi_{n,r}'(-1)]=0.
\end{aligned}\right.
\end{equation}
Therefore, $\psi_n =\psi_{n,s}+b_n (\psi_{n,BL}+\psi_{n,e})+a_n(\psi_{n, p}+\psi_{n,r})$ satisfies \eqref{streamBC} at $y=\pm 1$.
Solving the linear system \eqref{recoverBC1} gives
\begin{equation}\label{5-3-79}
a_n=-\frac{\psi_{n,BL}^+(1)}{\psi_{n,p}(1)}b_n\quad\text{and}\quad
b_n=\frac{\psi_{n,s}'(1)}{\psi_{n,BL}^+ (1)\frac{\psi_{n,p}'(1)+\psi_{n,r}'(1)}{\psi_{n,p}(1)}
-(\psi_{n,BL}^+)'(1)-\psi_{n,e}'(1)}.
\end{equation}
In order to get the estimates for $a_n$ and $b_n$, one of the key issue is to estimate $(\psi_{n,BL}^+)'(1)$. The straightforward computations show
\begin{equation}\nonumber
\begin{aligned}
(\psi_{n,BL}^+)'(1)=&-\beta C_{0,n,\Phi}\frac{d G_{n,\Phi}}{d \rho}(0)\\
=&\beta C_{0,n,\Phi}\left(\frac{|\hat{n}|}{\beta }G_{n,\Phi}(0)+C_-\int_\ell e^{-\lambda z}Ai\left(z+\lambda^2\right)\,dz\right),
\end{aligned}
\end{equation}
where $\lambda=\frac{|\hat{n}|}{\beta }C_-$ and $\ell$ is the contour $\ell=\left\{re^{i\frac{\pi}{6}}|r\ge 0\right\}$. Note that $\arg\lambda=-\frac{\pi}{6}$ and
\begin{equation}\nonumber
|\lambda|=\frac{|\hat{n}|}{\beta }=\left(\frac23\right)^\frac13(\Phi^{-\frac12}|\hat{n}|)^\frac23.
\end{equation}
Choose  $\epsilon_1\in (0,1)$ such that
\begin{equation}\nonumber
|\lambda|\le \min\{\epsilon,\frac{1}{48}\tilde{C}_0\}
\end{equation}
as long as $|\hat{n}|\le \epsilon_1\Phi^\frac12$,
where $\epsilon$ is the small constant indicated in Lemma \ref{airy-est}.
One can apply Lemma \ref{airy-est} to obtain
\begin{equation}\label{5-3-82}
\begin{aligned}
|(\psi_{n,BL}^+)'(1)|=&\left|\beta C_{0,n,\Phi}\left(\frac{|\hat{n}|}{\beta }G_{n,\Phi}(0)+C_-\int_\ell e^{-\lambda z}Ai\left(z+\lambda^2\right)\,dz\right)\right|\\
\ge&\left|\beta C_{0,n,\Phi}\int_\ell e^{-\lambda z}Ai\left(z+\lambda^2\right)\,d z\right|-\left|\beta C_{0,n,\Phi}\frac{|\hat{n}|}{\beta }G_{n,\Phi}(0)\right|\\
\ge&\beta \tilde{C}_0\left(\frac16-\frac{|\hat{n}|}{\tilde{C}_0\beta }\right)\\
\ge & \frac{1}{12}\beta \tilde{C}_0=:\kappa\beta .
\end{aligned}\end{equation}

For the term $\psi_{n,e}'(1)$, it follows from \eqref{5-3-74}-\eqref{5-3-75} and Lemma \ref{lemmaA2} that
\begin{equation}\label{5-3-83}
\begin{aligned}
|\psi_{n,e}'(1)|\le&C\left(\inte|\psi_{n,e}'|^2\,dy\right)^\frac14\left(\inte |\psi_{n,e}''|^2\,dy\right)^\frac14\le C|\Phi\hat{n}|^\frac{1}{12}.
\end{aligned}
\end{equation}

On the other hand, using Lemma \ref{lemmaA2} again gives
\begin{equation}\nonumber
\begin{aligned}
|\psi_{n,r}'(1)|\leq &C \left(\inte |\psi_{n,r}'|^2\,dy \right)^\frac14\left(\inte  |\psi_{n,r}''|^2\,dy \right)^\frac14\\
\leq& C|\Phi\hat{n}|^{-\frac34} \left(\inte \left|i\hat{n} U''(y)\psi_{n,p}\right|^2\,dy \right)^\frac12\\
\leq& C|\Phi\hat{n}|^\frac14 \left(\inte \left|\psi_{n,p}\right|^2\,dy \right)^\frac12\\
\leq& C|\Phi\hat{n}|^\frac14 \psi_{n,p}(1).
\end{aligned}
\end{equation}
Hence one has
\begin{equation}\label{5-3-85}
\begin{aligned}
\left|\psi_{n,BL}^+(1)\frac{\psi_{n,p}'(1)+\psi_{n,r}'(1)}{\psi_{n,p}(1)}\right|\leq C(|\hat{n}|+|\Phi\hat{n}|^\frac14)\leq \frac\kappa 4\beta,
\end{aligned}
\end{equation}
provided $\epsilon_1$ is sufficiently small, where $\kappa$ is the constant defined in \eqref{5-3-82}. For sufficiently large $\beta$, combining \eqref{5-3-79}-\eqref{5-3-85} yields
\begin{equation}\label{5-3-86}
\begin{aligned}
|b_n|\leq C\beta^{-1}|\Phi\hat{n}|^{-\frac12}\left(\inte |\BF_n|^2\,dy\right)^\frac12\leq C|\Phi\hat{n}|^{-\frac56}\left(\inte |\BF_n|^2\,dy\right)^\frac12
\end{aligned}
\end{equation}
and
\begin{equation}\label{5-3-87}
\begin{aligned}
|a_n|\leq &C|\Phi\hat{n}|^{-\frac56}|\psi_{n,p}(1)|^{-1}\left(\inte |\BF_n|^2\,dy\right)^\frac12.
\end{aligned}
\end{equation}

Next one can obtain the estimates of the boundary $b_n\psi_{n,BL}$. According to Lemma \ref{airy-est}, using $|\hat{n}|\leq \epsilon_1|\Phi|^\frac12$ yields
\begin{equation}\label{5-3-88}
\begin{aligned}
&|b_n|^2\inte \hat{n}^2|\psi_{n,BL}|^2+\left|\psi_{n,BL}'\right|^2\,dy\\
\le&C|b_n|^2\int_\frac14^1\hat{n}^2|\psi_{n,BL}^+|^2+\left|\psi_{n,BL}^+\right|^2+\left|(\psi_{n,BL}^+)'\right|^2\,dy\\
\le&C|b_n|^2\int_0^{\frac34\beta }\left(\hat{n}^2|G_{n,\Phi}(\rho)|^2+|G_{n,\Phi}(\rho)|^2+\left|\frac{d}{d\rho}G_{n,\Phi}(\rho)\right|^2\beta ^2\right)\beta^{-1}\,d \rho\\
\le&C | \Phi\hat{n}|^{-\frac53}\inte |\BF_n|^2\,dy\left(\int_0^R\beta +\beta ^{-1}\,d\rho+\int_R^{\frac34\beta }e^{-2\rho}(\beta +\beta ^{-1})\,d\rho\right)\\
\le&C|\Phi\hat{n}|^{-\frac43}\inte |\BF_n|^2\,dy.
\end{aligned}
\end{equation}
Similar computations give
\begin{equation}\label{5-3-89}
\begin{aligned}
&|b_n|^2\inte\left|\psi_{n,BL}^{(3)}\right|^2+\hat{n}^2\left|\psi_{n,BL}''\right|^2+\hat{n}^4\left|\psi_{n,BL}'\right|^2+\hat{n}^6\left|\psi_{n,BL}\right|^2\,dy
\le C\inte |\BF_n|^2\,dy.
\end{aligned}
\end{equation}

It follows from \eqref{5-3-73}-\eqref{5-3-76} and \eqref{5-3-86} that one has
\begin{equation}\label{5-3-90}
|b_n|^2\inte \hat{n}^2|\psi_{n,e}|^2+|\psi_{n,e}'|^2\,dy \le C|\Phi\hat{n}|^{-\frac53}\inte |\BF_n|^2\,dy
\end{equation}
and
\begin{equation}\label{5-3-91}
|b_n|^2\inte \left|\psi_{n,e}^{(3)}\right|^2+\hat{n}^2|\psi_{n,e}''|^2+\hat{n}^4|\psi_{n,e}'|^2+\hat{n}^6|\psi_{n,e}|^2\,dy \leq C|\Phi\hat{n}|^{-\frac23}\inte |\BF_n|^2\,dy.
\end{equation}
Meanwhile, the straightforward computations yield
\begin{equation}\label{5-3-92}\begin{aligned}
|a_n|^2\inte \hat{n}^2|\psi_{n,p}|^2+|\psi_{n,p}'|^2\,dy
\le &C|a_n|^2\inte \hat{n}^2|e^{\hat{n}y}+e^{-\hat{n}y}|^2+\hat{n}^2|e^{\hat{n}y}-e^{-\hat{n}y}|^2\,dy\\
\le &C|a_n|^2\inte \hat{n}^2(e^{2\hat{n}y}+e^{-2\hat{n}y})\,d y\\
\le &C|\Phi\hat{n}|^{-\frac43}\inte |\BF_n|^2\,dy
\end{aligned}\end{equation}
and
\be \label{5-3-93}\ba
|a_n|^2 \inte \left| \psi_{n,p}^{(3)}\right|^2 +\hat{n}^2|\psi_{n,p}''|^2 + \hat{n}^4 |\psi_{n,p}'|^2 + \hat{n}^6| \psi_{n,p} |^2 \,dy
\leq C\inte |\BF_n|^2\,dy.
\ea\ee
Finally, if follows from Lemma \ref{slip-est} that
\begin{equation}\label{5-3-94}\ba
~~|a_n|^2\inte \hat{n}^2|\psi_{n,r}|^2+|\psi_{n,r}'|^2\,dy
\leq &~~ C|a_n|^2 |\Phi\hat{n}|^{-\frac53}\inte \left|i\hat{n} U''(y)\psi_{n,p}\right|^2\,dy\\
\leq&~~ C|\Phi\hat{n}|^{-\frac43}\inte |\BF_n|^2\,dy
\ea\end{equation}
and
\begin{equation}\label{5-3-95}\ba
&~~|a_n|^2\inte \left|\psi_{n,r}^{(3)}\right|^2+\hat{n}^2|\psi_{n,r}''|^2+\hat{n}^4|\psi_{n,r}'|^2+\hat{n}^6|\psi_{n,r}|^2\,dy \\
\leq&~~ C|a_n|^2 |\Phi\hat{n}|^{-\frac23}\inte \left|i\hat{n} U''(y)\psi_{n,p}\right|^2\,dy\\
\leq&~~ C|\Phi\hat{n}|^{-\frac13}\inte |\BF_n|^2\,dy.
\ea\end{equation}
Combining the estimates \eqref{5-3-88}-\eqref{5-3-95} and Lemma \ref{slip-est} gives the estimates \eqref{est5-6} and \eqref{est5-7}.

{\it Step 5. The case with general $f_n$.}  From Steps 1-4, if $f_n$ is an even function,  we can construct the solution $\psi_n$  to the problem \eqref{stream} and \eqref{streamBC} in the form of \eqref{decompose} with $a_n^o=b_n^o=0$. Similarly,
if $f_n$ is an odd function, one can construct the solution $\psi_n$  to the problem \eqref{stream} and \eqref{streamBC} in the form of \eqref{decompose} with
$a_n^e=b_n^e=0$.
Since the boundary conditions \eqref{recoverBC1}, \eqref{recoverBC-1} are still the same, one can also get the estimates \eqref{5-3-86}-\eqref{5-3-95} for $a_n^o$, $b_n^o$, $\psi_{n,BL}^o$, $\psi_{n,e}^o$,  and $\psi_{n,r}^o$ in the same way. Therefore, for a given general $f_n$, the solutions associated with $f_{n}^e$ and $f_{n}^o$ satisfy the estimates \eqref{est5-6}-\eqref{est5-7}. Hence the solution $\psi_n$ satisfies \eqref{est5-6}-\eqref{est5-7} so that  the proof of Proposition \ref{mediumstream} is completed.
\end{proof}

\begin{pro} \label{mediumv}
Assume that $\Phi\gg1$ and $1 \le |n| \le \epsilon_1L\sqrt{\Phi}$, where $\epsilon_1$ is the constant indicated in Proposition \ref{mediumstream}. The corresponding velocity field  $\Bv_n$  satisfies
\be \label{est5-16}
\|\Bv_n\|_{L^2(\Omega)} \leq C |\Phi\hat{n}|^{-\frac23} \|\BF_n \|_{L^2(\Omega)}
\ee
and
\be \label{est5-17}
\|\Bv_n\|_{H^2(\Omega)} \leq C_2  \|\BF_n \|_{L^2(\Omega)},
\ee
where $C$ and $C_2$ are uniform constants independent of $\Phi$, $n$, and $\BF_n$.
\end{pro}

\begin{proof}
As shown in the proof of Lemma \ref{reg-velocity}, the solution $\Bv_n$ satisfies the elliptic equation \eqref{vorticity}. Applying the regularity theory for the elliptic equation with homogeneous boundary conditions (\cite{GT}) gives
\be \label{5-3-96}
\| \Bv_n\|_{H^2(\Omega) } \leq C \left\|\nabla(\omega_n e^{i\hat{n}x})\right\|_{L^2(\Omega)}+\| \Bv_n\|_{L^2(\Omega) } .
\ee
It follows from Proposition \ref{mediumstream} that one has
\be \label{5-3-97}
\|\Bv_n\|_{L^2(\Omega)}^2 \leq C L\inte |\psi_n'|^2+\hat{n}^2|\psi_n|^2 \,dy
\leq C |\Phi\hat{n}|^{-\frac43}\|\BF_n\|_{L^2(\Omega)}^2
\ee
and
\be \label{5-3-98}
\left\|\nabla(\omega_n e^{i\hat{n}x})\right\|_{L^2(\Omega)}^2 \leq
CL\inte \left|\psi_n^{(3)}\right|^2+\hat{n}^2|\psi_n''|^2+\hat{n}^4|\psi_n'|^2
+\hat{n}^6|\psi_n|^2\,dy \leq C\|\BF_n\|^2_{L^2(\Omega)}.
\ee
The estimate \eqref{5-3-97} is exactly \eqref{est5-16}. Substituting \eqref{5-3-97}-\eqref{5-3-98} into \eqref{5-3-96} yields \eqref{est5-17}. This finishes the proof of the proposition.
\end{proof}

\subsection{Uniform estimate for the case with large flux and high frequency}
In this subsection, we give the uniform estimate for the solutions of \eqref{stream}-\eqref{streamBC} with respect to the flux $\Phi$ when the flux is large and the frequency is high.

\begin{pro}\label{highstream}
Assume that $|n|\ge \epsilon_1L\sqrt{\Phi}\ge 1$ for some constant $\epsilon_1\in (0,1)$. Let $\psi_n$ be a solution to the problem \eqref{stream}--\eqref{streamBC}, then one has
\begin{equation}\label{est5-18}\begin{aligned}
\inte \Phi|\hat{n}|(1-y^2)\left|\psi_n'\right|^2+\Phi|\hat{n}|^3(1-y^2)|\psi_n|^2\,dy&\\
+\inte |\psi_n''|^2+\hat{n}^2\left|\psi_n'\right|^2 +\hat{n}^4|\psi_n|^2\,dy&\le C|\hat{n}|^{-2}\inte |\BF_n|^2\,dy.
\end{aligned}\end{equation}
\end{pro}

\begin{proof}
It follows from \eqref{2-2-4}, \eqref{2-2-6}, Lemma \ref{lemmaHLP} and integration by parts that one has
\begin{equation}\label{5-4-1}
\inte |\psi_n''|^2+ \hat{n}^2\left|\psi_n'\right|^2+ \hat{n}^4|\psi_n|^2\,dy
\le C\left| \int_{-1}^1 f_n \overline{\psi}_n  \, dy  \right| +C\Phi |\hat{n}|\inte \left|y\psi_n'\bar{\psi}_n\right|\,dy
\end{equation}
and
\begin{equation}\label{5-4-2}
\begin{aligned}
\Phi |\hat{n}|\inte \left|\psi_n\right|^2+ \left|\psi_n'\right|^2(1-y^2)+\hat{n}^2 |\psi_n|^2(1-y^2)\,dy
\le C\left| \int_{-1}^1 f_n \overline{\psi}_n \, dy  \right| .
\end{aligned}\end{equation}
Herein,
\be\nonumber
\ba
& \Phi |\hat{n}| \int_{-1}^1 | y \psi_n' \bar{\psi}_n| \, dy \\
\leq & C \Phi |\hat{n}| \left( \int_{-1}^1 |\psi_n'|^2 \, dy  \right)^{\frac12} \left( \int_{-1}^1 |\psi_n|^2 \, dy \right)^{\frac12} \\
\leq & C \Phi |\hat{n}| \left[\left( \int_{-1}^1 |\psi_n''|^2 \, dy \right)^{\frac16} \left( \int_{-1}^1 (1 - y^2) |\psi_n'|^2 \, dy \right)^{\frac13}+\left( \int_{-1}^1 (1 - y^2) |\psi_n'|^2 \, dy \right)^{\frac12}\right]\\
&\ \ \ \ \ \ \ \cdot \left[\left( \int_{-1}^1 |\psi_n'|^2 \, dy \right)^{\frac16} \left(  \int_{-1}^1 (1 - y^2) |\psi_n|^2 \, dy \right)^{\frac13} +\left( \int_{-1}^1 (1 - y^2) |\psi_n|^2 \, dy \right)^{\frac12}\right]\\
\leq & C \Phi |\hat{n}| \left[\left( \int_{-1}^1 |\psi_n''|^2 \, dy \right)^{\frac16} \left( \Phi|\hat{n}|\right)^{-\frac13}\left| \int_{-1}^1 f_n \overline{\psi}_n \, dy  \right|^\frac13+\left( \Phi|\hat{n}|\right)^{-\frac12}\left| \int_{-1}^1 f_n \overline{\psi}_n \, dy  \right|^\frac12\right]\\
&\ \ \ \ \ \ \ \cdot \left[\left( \int_{-1}^1 \hat{n}^2 |\psi_n'|^2 \, dy \right)^{\frac16}\left( \Phi|\hat{n}|^4\right)^{-\frac13}\left| \int_{-1}^1 f_n \overline{\psi}_n \, dy  \right|^\frac13+\left( \Phi|\hat{n}|^3\right)^{-\frac12}\left| \int_{-1}^1 f_n \overline{\psi}_n \, dy  \right|^\frac12\right].
\ea\ee
Using H\"{o}lder inequalities gives
\be \label{5-4-3}
\ba
& \Phi |\hat{n}| \int_{-1}^1 | y \psi_n' \bar{\psi}_n| \, dy \\
\leq & \frac14 \int_{-1}^1 |\psi_n''|^2 \, dy + \frac14 \int_{-1}^1 \hat{n}^2 |\psi_n'|^2 \, dy +
C (\Phi^{\frac12} |\hat{n}|^{-1}+\Phi^{\frac15} |\hat{n}|^{-1}+|\hat{n}|^{-1}) \left|\int_{-1}^1 f_n \bar{\psi}_n \, dy  \right|\\
\leq & \frac14 \int_{-1}^1 |\psi_n''|^2 \, dy + \frac14 \int_{-1}^1 \hat{n}^2 |\psi_n'|^2 \, dy +
C \left|\int_{-1}^1 f_n \bar{\psi}_n \, dy  \right|.
\ea
\ee
Taking \eqref{5-4-3} into \eqref{5-4-1} yields
\be \nonumber
\ba
& \inte |\psi_n''|^2+ \hat{n}^2\left|\psi_n'\right|^2+ \hat{n}^4|\psi_n|^2\,dy\le  C\left|\int_{-1}^1 f_n \bar{\psi}_n \, dy  \right| \\
\leq & C|\hat{n}|^{-1}  \left( \int_{-1}^1 |\BF_n|^2 \, dy  \right)^{\frac12} \left( \int_{-1}^1 \hat{n}^4 |\psi_n|^2 + \hat{n}^2|\psi_n'|^2 \, dy \right)^{\frac12}.
\ea
\ee
By Young's inequality,
\be \label{5-4-5}
 \inte |\psi_n''|^2+ \hat{n}^2\left|\psi_n'\right|^2+ \hat{n}^4|\psi_n|^2\,dy
\leq C |\hat{n}|^{-2} \int_{-1}^1 |\BF_n|^2 \, dy .
\ee
Substituting \eqref{5-4-5} into \eqref{5-4-2} gives the inequality \eqref{est5-18}. This finishes the proof of the proposition.
\end{proof}

\begin{pro}\label{highv}
Assume that $|n|\ge \epsilon_1L\sqrt{\Phi}\ge 1$ for some constant $\epsilon_1\in (0,1)$. The solution $\Bv_n$ satisfies
\begin{equation}\nonumber
\|\Bv_n\|_{L^2(\Omega)}\le C\Phi^{-1}\|\BF_n\|_{L^2(\Omega)},
\end{equation}
\begin{equation}\label{est5-21}
\|\Bv_n\|_{H^2(\Omega)}\le C(1+\Phi^\frac14)\|\BF_n\|_{L^2(\Omega)},
\end{equation}
and
\begin{equation}\nonumber
\|\Bv_n\|_{H^\frac53(\Omega)}\le C_3\|\BF_n\|_{L^2(\Omega)},
\end{equation}
for some positive constants $C$ and $C_3$ independent of flux $\Phi$, $n$, and $\BF_n$.
\end{pro}
\begin{proof}First, by virtue of \eqref{est5-18}, one has
\begin{equation}\nonumber
\begin{aligned}
\|\Bv_n\|_{L^2(\Omega)}\le&CL^\frac12\left(\inte \hat{n}^2|\psi_n|^2+\left|\psi_n'\right|^2 \,dy\right)^\frac12\\
\le&C\hat{n}^{-2}L^\frac12\left(\inte |\BF_n|^2\,dy\right)^\frac12\le C\Phi^{-1}\|\BF_n\|_{L^2(\Omega)}.
\end{aligned}
\end{equation}
Using \eqref{est-stokes} yields
\begin{equation}\label{5-4-8}
\begin{aligned}
\|\Bv_n \|_{H^2(\Omega)}\le& C(\|\BF_n\|_{L^2(\Omega)} +\Phi \|(1 - y^2) \partial_x \Bv_n\|_{L^2(\Omega)} + \Phi \|v_{2,n}e^{i\hat{n} x}\|_{L^2(\Omega)} +  \|\Bv_n\|_{H^1(\Omega)}).
\end{aligned}\end{equation}
It follows from \eqref{est5-18} that one has
\begin{equation}\label{5-4-9}
\begin{aligned}
&\Phi\|(1-y^2)\partial_x \Bv_n \|_{L^2(\Omega)}\\
\le& C\Phi L^\frac12\left(\inte \hat{n}^4(1-y^2)^2|\psi_n|^2+\hat{n}^2(1-y^2)^2\left|\psi_n'\right|^2\,dy\right)^\frac12\\
\le&C|\Phi\hat{n}|^\frac12L^\frac12\left(\inte \Phi|\hat{n}|^3(1-y^2)|\psi_n|^2+\Phi|\hat{n}|(1-y^2)\left|\psi_n'\right|^2\,dy\right)^\frac12\\
\le&C\Phi^\frac12|\hat{n}|^{-\frac12}L^\frac12\left(\inte |\BF_n|^2\,dy\right)^\frac12\le C\Phi^\frac14\|\BF_n\|_{L^2(\Omega)}
\end{aligned}\end{equation}
and
\begin{equation}\label{5-4-10}
\begin{aligned}
\Phi\|v_{2,n}e^{i\hat{n}x}\|_{L^2(\Omega)}\le& C\Phi\|\Bv_n\|_{L^2(\Omega)}\le C\|\BF_n\|_{L^2(\Omega)}.
\end{aligned}\end{equation}
By Poincar\'{e}'s inequality and Proposition \ref{highstream}, one has
\begin{equation}\label{5-4-11}
\begin{aligned}
\|\Bv_n \|_{H^1(\Omega)}\le& C\|\nabla\Bv_n \|_{L^2(\Omega)}
\le CL^\frac12\left(\inte \hat{n}^4|\psi_n|^2+\hat{n}^2\left|\psi_n'\right|^2 +|\psi_n''|^2\,dy\right)^\frac12\\
\le& C \Phi^{-\frac12}\|\BF_n\|_{L^2(\Omega)}.
\end{aligned}\end{equation}
Combining \eqref{5-4-8}-\eqref{5-4-11} yields \eqref{est5-21}. Finally, using interpolation between $H^2(\Omega)$ and $H^1(\Omega)$ gives
\begin{equation}\nonumber
\|\Bv_n \|_{H^\frac53(\Omega)}\le C\|\Bv_n \|_{H^1(\Omega)}^\frac13\|\Bv_n \|_{H^2(\Omega)}^\frac23\le C_2\|\BF_n \|_{L^2(\Omega)}.
\end{equation}
Hence the proof of the proposition is completed.
\end{proof}
Combining Propositions \ref{smallflux}, \ref{0mode}, \ref{mediumv}, and \ref{highv} together and using density argument give the existence of the solutions for the problem \eqref{model12} and \eqref{model11'}. This, together with the a priori estimates obtained in Propositions \ref{smallflux}, \ref{0mode}, \ref{mediumv}, and \ref{highv} finishes the proof of Theorem \ref{thm1}.

%%%%%%%%%%%%%%%%%%%%%%%%%%Nonlinear%%%%%%%%%%%%%%%%%%%%%%%%%%%%%%%%%%%%%%%%%%%%%%%%%%%%

\section{The nonlinear problem}\label{sec-nonlinear}
In this section, we prove the existence and uniqueness of  solution to the nonlinear
problem \eqref{model11}-\eqref{model11'}. As proved in Theorem \ref{thm1}, for any  external force $\BF(x,y)\in L^2(\Omega)$ there exists a solution $\Bv\in H^2(\Omega)$ to the linearized problem \eqref{model12} and \eqref{model11'}. Denote $\Bv=\T(\BF)$.

Set $Y=H^\frac53(\Omega)$. $\T(\BF)$ satisfies
\begin{equation}\nonumber
\|\T(\BF)\|_{Y}\le C\|\BF\|_{L^2(\Omega)}.
\end{equation}
For any $\Bz_1,\Bz_2 \in Y$, the bilinear form
\begin{equation}\nonumber
\mathcal{B}(\Bz_1,\Bz_2):=\mathcal{T}(-\Bz_1\cdot\nabla\Bz_2),
\end{equation}
is well defined and $\mathcal{B}(\Bz_1,\Bz_2)$ satisfying
\begin{equation}\nonumber
\|\B(\Bz_1,\Bz_2)\|_{Y}\le C\|\Bz_1\|_{L^{12}(\Omega)}\|\nabla\Bz_2\|_{L^\frac{12}{5}(\Omega)}\le C\|\Bz_1\|_{Y}\|\Bz_2\|_{Y}.
\end{equation}
Hence it follows from Lemma \ref{lemmaA6} that if
$\|\BF\|_{L^2(\Omega)}\le \varepsilon$
for some sufficiently small $\varepsilon$, the equation
\begin{equation}\nonumber
\Bv=\T(\BF)+\B(\Bv,\Bv)
\end{equation}
has a unique solution $\Bv\in Y$ satisfying $\|\Bv\|_Y\le C\|\BF\|_{L^2(\Omega)}$. In fact, $\Bv$ is a strong solution satisfying that
\begin{equation}\nonumber
\|\Bv\|_{H^2(\Omega)}\le C(1+\Phi^\frac14)(\|\BF\|_{L^2(\Omega)}+\|\Bv\|_Y^2)\le C(1+\Phi^\frac14)\|\BF\|_{L^2(\Omega)}.
\end{equation}
This finishes the proof of part(a) of Theorem 1.2.

Next, let us consider the existence and uniqueness of strong solutions to \eqref{model11}-\eqref{model11'}, when the external force $\BF$ is large in $L^2(\Omega)$.  Let $\Bv$ be a periodic solution of the nonlinear system \eqref{model11}-\eqref{model11'}. Then for any $n\in \mathbb{Z}$, the stream function $\psi_n$ of $\Bv_n$ satisfies the following nonlinear system
\begin{equation}\label{model71}
\begin{aligned}
&-i\hat{n}U''\psi_n+i\hat{n}U\left(\frac{d^2}{dy^2}-\hat{n}^2\right)\psi_n-\left(\frac{d^2}{dy^2}-\hat{n}^2\right)^2\psi_n\\
=&i\hat{n}F_{2,n}-\frac{d}{dy}F_{1,n}-\frac{d}{dy}\left(\sum\limits_{m\in \mathbb{Z}} v_{2,n-m}\omega_{m}\right)-i\hat{n}\sum\limits_{m\in \mathbb{Z}} v_{1,n-m}\omega_{m},
\end{aligned}
\end{equation}
and the boundary conditions \eqref{streamBC}.

\begin{pro}
\label{largeF}
Assume that $\BF\in L^2(\Omega)$ and $L\leq L_0$. There exists a constant $\Phi_0\ge 1$ such that for every $\Phi\ge \Phi_0$, if
\begin{equation}\nonumber
\|\BF\|_{L^2(\Omega)}\le \Phi^\frac{1}{32},
\end{equation}
the nonlinear problem \eqref{model11}-\eqref{model11'} admits a unique solution $\Bv\in H^2(\Omega)$ satisfying
\begin{equation}\nonumber
\left(\left\|\Bv_0\right\|_{H^2(\Omega)}^2+\sum\limits_{n\neq0}\left\|\Bv_n\right\|_{H^\frac53(\Omega)}^2\right)^\frac12
\le C\|\BF\|_{L^2(\Omega)}
\end{equation}
and
\begin{equation}\nonumber
\|v_2\|_{L^2(\Omega)}\leq \Phi^{-\frac{7}{12}}, \quad
\left\|\Bv\right\|_{H^2(\Omega)}\le C\Phi^\frac14\|\BF\|_{L^2(\Omega)},
\end{equation}
where $C$ is a uniform constant independent of the flux $\Phi$ and $L$.
\end{pro}

\begin{proof}
The proof is divided into 4 steps. The existence of the solutions is established by the iteration method.

 {\em Step 1: Iteration scheme.}  Given $\BF\in L^2(\Omega)$, the linear problem \eqref{stream}-\eqref{streamBC} admits a unique solution $\psi_n^0$ for each $n$. The corresponding velocity field is denoted by
\begin{equation}\nonumber
\Bv^0=\sum\limits_{n\in\Z}v_{1,n}^0e^{i\hat{n}x}\Be_1+v_{2,n}^0e^{i\hat{n}x}\Be_2  \text{ with }
v^0_{2,n}=i\hat{n}\psi_n^0~~\text{ and }v^0_{1,n}=-\frac{d}{dy}\psi_n^0.
\end{equation}
For every $j\ge 0$, let $\psi^{j+1}_n$ be the solution to the iteration problem
\begin{equation}\label{model72}
\left\{\begin{aligned}
-i\hat{n}U''\psi_n^{j+1}&+i\hat{n}U\left(\frac{d^2}{dy^2}-\hat{n}^2\right)\psi^{j+1}_n-\left(\frac{d^2}{dy^2}-\hat{n}^2\right)^2\psi^{j+1}_n\\
=&i\hat{n} (F_{2,n}+F^j_{2,n})-\frac{d}{dy}(F_{1,n}+F_{1,n}^j),\\
\psi_n^{j+1}(\pm1)=&\frac{d}{dy}\psi_n^{j+1}(\pm1)=0,
\end{aligned}\right.
\end{equation}
where
\begin{equation}\nonumber
\begin{aligned}
F^j_{2,n}=&-\sum\limits_{m\in \mathbb{Z}} v_{1,n-m}^j \omega_{m}^j ,~~\ F^j_{1,n}=\sum\limits_{m\in \mathbb{Z}} v_{2,n-m}^j \omega_{m}^j, \\
\end{aligned}
\end{equation}
and
\begin{equation}\nonumber
v^j_{2,n}=i\hat{n}\psi_n^j,~~v^j_{1,n}=-\frac{d}{dy}\psi_n^j,~~\omega^j_{n}=\left(\frac{d^2}{dy^2}-\hat{n}^2\right)\psi_n^j.
\end{equation}
For convenience, we define the projection operator
\begin{equation}\label{projection}
\Q \Bv=\sum\limits_{n\neq 0}(v_{1,n}\Be_1+v_{2,n}\Be_2)e^{i\hat{n}x}.
\end{equation}

Assume that  $\|\BF\|_{L^2(\Omega)}\le \Phi^\frac{1}{32}$ and set
\begin{equation}\nonumber
\mathcal{J}=\left\{\Bv=\sum\limits_{n\in\Z} \Bv_n
\left|\begin{array}{l}v_{2,0}=0,~~\ \ \left\|\Q\Bv\right\|_{L^2(\Omega)}\le\Phi^{-\frac{7}{12}},\\ \left(\left\|\Bv_0\right\|_{H^2(\Omega)}^2+\left\|\Q\Bv\right\|_{H^\frac53(\Omega)}^2\right)^\frac12\le 2C_4\|\BF\|_{L^2(\Omega)}
\end{array}\right.
\right\},
\end{equation}
where $C_4=\max\{C_1,C_2,C_3\}$ and $C_i(i=1,2,3)$ are positive constants appeared in Propositions \ref{0mode}, \ref{mediumv}, and \ref{highv}, respectively.

{\em Step 2: Mathematical induction.} According to the linear estimates obtained in Propositions \ref{0mode}, \ref{mediumv}, and \ref{highv}, one has $\Bv^0\in \mathcal{J}$ immediately. Assume that $\Bv^j\in \mathcal{J}$, we next prove that $\Bv^{j+1}\in \mathcal{J}$. First, it follows from Propositions \ref{0mode}, \ref{mediumv} and \ref{highv} and the assumption $v^j_{2,0}=0$ that
\begin{equation}\begin{aligned}\label{eq706}
&\left(\left\|\Bv_0^{j+1}\right\|_{H^2(\Omega)}^2+\left\|\Q\Bv^{j+1}\right\|_{H^\frac53(\Omega)}^2\right)^\frac12\\
\le& C_4\left(\left\|F_{1,0}+F_{1,0}^j\right\|_{L^2(\Omega)}^2+\sum\limits_{n\neq 0}\left\|F_{2,n}+F_{2,n}^j\right\|_{L^2(\Omega)}^2+\left\|F_{1,n}+F_{1,n}^j\right\|_{L^2(\Omega)}^2\right)^\frac12\\
\le& \sqrt{2}C_4\left\|\BF\right\|_{L^2(\Omega)}
+C\left(\sum\limits_{n\in \mathbb{Z} }\left\|F^{j}_{1,n}\right\|_{L^2(\Omega)}^2+\sum\limits_{n\neq 0}\left\|F^{j}_{2,n}\right\|_{L^2(\Omega)}^2\right)^\frac12\\
\le&\sqrt{2}C_4\left\|\BF\right\|_{L^2(\Omega)}+C\left[\sum\limits_{n\in \mathbb{Z} } \left\| \sum\limits_{m\neq 0,n}v^j_{2,n-m}\omega^j_{m}\right\|_{L^2(\Omega)}^2 + \sum\limits_{n \neq 0} \left\|\sum\limits_{m\neq 0,n}v^j_{1,n-m}\omega^j_{m}\right\|_{L^2(\Omega)}^2 \right.\\
&\left.+\sum\limits_{n\neq 0}\left(\left\|v^j_{2,n}\omega^j_{0}\right\|_{L^2(\Omega)}^2+\left\|v^j_{1,n}\omega^j_{0}\right\|_{L^2(\Omega)}^2+\left\|v^j_{1,0}\omega^j_{n}\right\|_{L^2(\Omega)}^2\right)\right]^\frac12.
\end{aligned}\end{equation}
Using Hausdorff-Young inequality and the interpolation inequalities yields
\begin{equation}\begin{aligned}\label{eq707}
&\left(\sum\limits_{n\in \mathbb{Z} }\left\| \sum\limits_{m\neq 0,n}v^j_{2,n-m}\omega^j_{m}\right\|_{L^2(\Omega)}^2\right)^\frac12=
\left(\sum\limits_{n\in \mathbb{Z} } \left\| \sum\limits_{m\neq 0,n}(\Q v_2^j)_{n-m}(\Q \omega^j)_m\right\|_{L^2(\Omega)}^2\right)^\frac12\\
 \leq & \left(\sum\limits_{n\in \mathbb{Z} }\left\|(\Q v_2^j\cdot\Q \omega^j)_n\right\|_{L^2(\Omega)}^2\right)^\frac12=\left\|\Q v_2^j\cdot\Q \omega^j\right\|_{L^2(\Omega)} \le\left\|\Q v_2^j\right\|_{L^6(\Omega)}\left\|\Q \omega^j\right\|_{L^3(\Omega)}\\
\le& C\left\|\Q \Bv^j\right\|_{L^2(\Omega)}^\frac{2}{5}\left\|\Q \Bv^j\right\|_{H^\frac53(\Omega)}^\frac{3}{5}\left\|\Q \Bv^{j}\right\|_{H^\frac53(\Omega)}\\
\le& C\Phi^{-\frac{7}{30}}\|\BF\|_{L^2(\Omega)}^\frac85.
\end{aligned}\end{equation}
Similarly, it also holds that
\begin{equation}\label{eq708}
\left(\sum\limits_{n\neq 0}\left\| \sum\limits_{m\neq 0,n}v^j_{1,n-m}\omega^j_{m}\right\|_{L^2(\Omega)}^2\right)^\frac12\leq \left\|\Q v_1^j\right\|_{L^6(\Omega)}\left\|\Q \omega^j\right\|_{L^3(\Omega)}
\le C\Phi^{-\frac{7}{30}}\|\BF\|_{L^2(\Omega)}^\frac85.
\end{equation}
Using H\"{o}lder inequality and Sobolev's embedding inequalities yields
\begin{equation}\begin{aligned}\label{eq709}
&\left(\sum\limits_{n\neq 0}\left\|v^j_{2,n}\omega^j_{0}\right\|_{L^2(\Omega)}^2+\sum\limits_{n\neq 0}\left\|v^j_{1,n}\omega^j_{0}\right\|_{L^2(\Omega)}^2\right)^\frac12\le C\left(\sum\limits_{n\neq 0}\left\|\Bv^{j}_n\right\|_{L^6(\Omega)}^2\left\|\omega^j_{0}\right\|_{L^3(\Omega)}^2\right)^\frac12\\
\le&C\left\|\Bv^{j}_{0}\right\|_{H^2(\Omega)}\|\Q\Bv^j\|_{H^1(\Omega)}\le C\left\|\Bv^{j}_{0}\right\|_{H^2(\Omega)}\left\|\Q\Bv^{j}\right\|_{L^2(\Omega)}^\frac{2}{5}\left\|\Q\Bv^{j}\right\|_{H^\frac53(\Omega)}^\frac{3}{5}\\
\le& C\Phi^{-\frac{7}{30}}\|\BF\|_{L^2(\Omega)}^\frac{8}{5}
\end{aligned}\end{equation}
and
\begin{equation}\begin{aligned}\label{eq710}
&\left(\sum\limits_{n\neq 0}\left\|v^j_{1,0}\omega^j_{n}\right\|_{L^2(\Omega)}^2\right)^\frac12
\le  C\|\Bv_{0}^{j}\|_{L^\infty(\Omega)}\|\Q\Bv^{j}\|_{H^1(\Omega)}\\
\le& C\|\Bv_{0}^{j}\|_{H^2(\Omega)}\|\Q\Bv^{j}\|_{L^2(\Omega)}^\frac25\|\Q\Bv^{j}\|_{H^\frac53(\Omega)}^\frac35\le C\Phi^{-\frac{7}{30}}\|\BF\|_{L^2(\Omega)}^\frac85.
\end{aligned}\end{equation}
Combining the estimates \eqref{eq706}-\eqref{eq710}, together with the assumption $\|\BF\|_{L^2(\Omega)}\le \Phi^\frac{1}{32}$ gives
\begin{equation}\nonumber
\begin{aligned}
&\left(\left\|\Bv_0^{j+1}\right\|_{H^2(\Omega)}^2+\left\|\Q\Bv^{j+1}\right\|_{H^\frac53(\Omega)}^2\right)^\frac12\\
\le& \sqrt{2}C_4\|\BF\|_{L^2(\Omega)}+C\Phi^{-\frac{103}{480}}\|\BF\|_{L^2(\Omega)}.
\end{aligned}\end{equation}
Moreover, the estimates \eqref{eq706}-\eqref{eq710} together with  Propositions \ref{mediumv} and \ref{highv} yield that
\begin{equation}\nonumber
\begin{aligned}
&\left\|\Q\Bv^{j+1}\right\|_{L^2(\Omega)}\le C\Phi^{-\frac23}\left(\sum\limits_{n\neq0}\left\|F_{2,n}+F_{2,n}^j\right\|_{L^2(\Omega)}^2+\left\|F_{1,n}+F_{1,n}^j\right\|_{L^2(\Omega)}^2\right)^\frac12\\
\le& C\Phi^{-\frac{2}{3}}\|\BF\|_{L^2(\Omega)}+C\Phi^{-\frac{423}{480}}\|\BF\|_{L^2(\Omega)}.
\end{aligned}\end{equation}
There exists a positive constant $\Phi_0\ge 1$ such that for all $\Phi\ge \Phi_0$, the solution $\Bv^{j+1}$ satisfies
\begin{equation}\label{eq712}
\begin{aligned}
\left(\left\|\Bv_0^{j+1}\right\|_{H^2(\Omega)}^2+\left\|\Q\Bv^{j+1}\right\|_{H^\frac53(\Omega)}^2\right)^\frac12\le2C_4\|\BF\|_{L^2{(\Omega)}}
\end{aligned}\end{equation}
and
\begin{equation}\nonumber\begin{aligned}
&\left\|\Q\Bv^ {j+1}\right\|_{L^2(\Omega)}\le \Phi^{-\frac{7}{12}}.
\end{aligned}\end{equation}
Hence $\Bv^{j+1}\in \mathcal{J}$. By mathematical induction, 	$\Bv^j \in \mathcal{J}$ for every $j\in \mathbb{N}$ and $\{\Bv^j\}_{j\ge0}$ is a bounded sequence in
\begin{equation}\nonumber
\mathcal{J}_0=\left\{\Bv=\sum\limits_{n\in\Z} \Bv_n
\left|\left(\left\|\Bv_0\right\|_{H^2(\Omega)}^2+\left\|\Q\Bv\right\|_{H^\frac53(\Omega)}^2\right)^\frac12<\infty
\right.\right\}.
\end{equation}
Therefore, there exists a function $\Bv\in \mathcal{J}_0$ such that $\Bv^j\rightharpoonup \Bv$ in $\mathcal{J}_0$ and
\begin{equation}\label{eq715}
\left(\left\|\Bv_0\right\|_{H^2(\Omega)}^2+\left\|\Q\Bv\right\|_{H^\frac53(\Omega)}^2\right)^\frac12\le 2C_4\|\BF\|_{L^2(\Omega)}.
\end{equation}
Since $\psi^{j+1}_n$ is the solution to the problem \eqref{model72}, $\Bv^{j+1}$ satisfies
\begin{equation}\nonumber
\mathrm{curl}(-\Delta \Bv^{j+1}+\BU\cdot \nabla \Bv^{j+1}+\Bv^{j+1}\cdot \nabla \BU+\Bv^{j}\cdot\nabla\Bv^{j}-\BF)=0.
\end{equation}
Taking the limit for $j$ in the above equation yields
\begin{equation}\nonumber
\mathrm{curl}(-\Delta \Bv+\BU\cdot \nabla \Bv+\Bv\cdot \nabla \BU+\Bv\cdot\nabla\Bv-\BF)=0.
\end{equation}
Therefore there exists a function $P$ with $\nabla P\in L^2(\Omega)$ such that
\begin{equation}\label{eq718}
-\Delta \Bv+\BU\cdot \nabla \Bv+\Bv\cdot \nabla \BU+\nabla P=-\Bv\cdot\nabla\Bv+\BF.
\end{equation}

{\em Step 3: $H^2$-regularity. } For any $n\in\Z$, the stream function $\psi_n$ of $\Bv_n$ satisfies the nonlinear problem \eqref{model71}. According to Propositions \ref{mediumv} and \ref{highv}, the solution $\Bv$ satisfies the following estimates
\begin{equation}\nonumber
\begin{aligned}
\left\|\Q\Bv\right\|_{L^2(\Omega)}\leq& C\Phi^{-\frac23}\left\|\Q\BF\right\|_{L^2(\Omega)}
+C\Phi^{-\frac23}\left[\left(\sum\limits_{n\in \mathbb{Z}}\left\| \sum\limits_{m\neq 0,n}v_{2,n-m}\omega_{m}\right\|_{L^2(\Omega)}^2\right)^\frac12\right.\\
&+\left(\sum\limits_{n\neq 0}\left\|\sum\limits_{m\neq 0,n}v_{1,n-k}\omega_{k}\right\|_{L^2(\Omega)}^2\right)^\frac12+\left(\sum\limits_{n\neq 0}\left\|v_{2,n}\omega_{0}\right\|_{L^2(\Omega)}^2\right)^\frac12\\
&\left.+\left(\sum\limits_{n\neq 0}\left\|v_{1,n}\omega_{0}\right\|_{L^2(\Omega)}^2\right)^\frac12
+\left(\sum\limits_{n\neq 0}\left\|v_{1,0}\omega_{n}\right\|_{L^2(\Omega)}^2\right)^\frac12\right].
\end{aligned}\end{equation}
Using Hausdorff-Young inequality and Sobolev's embedding inequality gives
\begin{equation}\nonumber
\begin{aligned}
\left(\sum\limits_{n\in \mathbb{Z} }\left\| \sum\limits_{m\neq 0,n}v_{2,n-m}\omega_{m}\right\|_{L^2(\Omega)}^2\right)^\frac12\le \left\|\Q \Bv\cdot\Q \omega\right\|_{L^2(\Omega)} \le C\left\|\Q\Bv\right\|_{H^\frac53(\Omega)}^2\le C\|\BF\|_{L^2(\Omega)}^2
\end{aligned}\end{equation}
and
\begin{equation}\nonumber
\begin{aligned}
\left(\sum\limits_{n\neq 0}\left\|v_{2,n}\omega_{0}\right\|_{L^2(\Omega)}^2\right)^\frac12
\le&C\left\|\Bv_{0}\right\|_{H^2(\Omega)}\left\|\Q\Bv\right\|_{H^\frac53(\Omega)}\le C\|\BF\|_{L^2(\Omega)}^2.
\end{aligned}\end{equation}
Since the other terms can be estimated similarly, one finally obtains
\begin{equation}\begin{aligned}\label{eq722}
\left\|\Q\Bv\right\|_{L^2(\Omega)}
\le  C\Phi^{-\frac23}\|\BF\|_{L^2(\Omega)}^2\le \Phi^{-\frac{7}{12}},
\end{aligned}\end{equation}
provided that $\Phi\ge \Phi_0$ is sufficiently large. Moreover, it follows from Propositions \ref{mediumv} and \ref{highv} that
\begin{equation}\begin{aligned}\label{eq723}
\left\|\Q\Bv\right\|_{H^2(\Omega)}
\le&C\Phi^\frac14\left\|\Q\BF\right\|_{L^2(\Omega)}+C\Phi^{\frac14}\left[\left(\sum\limits_{n\in \mathbb{Z} }\left\| \sum\limits_{m\neq 0,n}v_{2,n-m}\omega_{m}\right\|_{L^2(\Omega)}^2\right)^\frac12\right.\\
&+\left(\sum\limits_{n\neq 0}\left\|\sum\limits_{m\neq 0,n}v_{1,n-k}\omega_{k}\right\|_{L^2(\Omega)}^2\right)^\frac12+\left(\sum\limits_{n\neq 0}\left\|v_{2,n}\omega_{0}\right\|_{L^2(\Omega)}^2\right)^\frac12\\
&\left.+\left(\sum\limits_{n\neq 0}\left\|v_{1,n}\omega_{0}\right\|_{L^2(\Omega)}^2\right)^\frac12
+\left(\sum\limits_{n\neq 0}\left\|v_{1,0}\omega_{n}\right\|_{L^2(\Omega)}^2\right)^\frac12\right].
\end{aligned}\end{equation}
With the aid of the estimates \eqref{eq715} and \eqref{eq722}, one can estimate the terms on the right hand side of \eqref{eq723} in a similar way to \eqref{eq712}. Hence, it holds that
\begin{equation}\nonumber
\left\|\Q\Bv\right\|_{H^2(\Omega)}
\leq C\Phi^{\frac14}\left\|\BF\right\|_{L^2(\Omega)}+ C\Phi^{\frac{17}{480}}\left\|\BF\right\|_{L^2(\Omega)}.
\end{equation}
This, together with \eqref{eq715}, implies that $\Bv$ is a strong solution of \eqref{eq718} and satisfies that
\begin{equation}\nonumber
\left\|\Bv\right\|_{H^2(\Omega)}\le C\Phi^{\frac14}\left\|\BF\right\|_{L^2(\Omega)}.
\end{equation}

{\em Step 4: Uniqueness.} To prove the uniqueness of the solution $\Bv$, we assume that $\Bv^j\in \mathcal{J}(j=1,2)$ are two solutions of the nonlinear problem \eqref{model11}-\eqref{model11'} satisfying
\begin{equation}\nonumber
\left(\left\|\Bv^{j}_0\right\|_{H^2(\Omega)}^2+\left\|\Q\Bv^{j}\right\|_{H^\frac53(\Omega)}^2\right)^\frac12\le C\left\|\BF\right\|_{L^2(\Omega)}\le C\Phi^\frac{1}{32}, ~~j=1,2.
\end{equation}
Then the solutions $\Bv^j(j=1,2)$ satisfy the estimate \eqref{eq722}. In addition, for each $n$, the difference of the stream functions $\tilde{\psi}_n:=\psi_n^1-\psi_n^2$ satisfies the equation
\begin{equation}\nonumber
\begin{aligned}
&-i\hat{n}U''\tilde{\psi}_n+i\hat{n}U\left(\frac{d^2}{dy^2}-\hat{n}^2\right)\tilde{\psi}_n-\left(\frac{d^2}{dy^2}-\hat{n}^2\right)^2\tilde{\psi}_n\\
=&i\hat{n}(F^1_{2,n}-F^2_{2,n})
-\frac{d}{dy}(F^1_{1,n}-F^2_{1,n}),
\end{aligned}
\end{equation}
where
\[
F_{1,n}^i = \sum\limits_{m\in \mathbb{Z}}v_{2,n-m}^i\omega_{m}^i\quad \text{and}\quad F_{2,n}^i=\sum\limits_{m\in \mathbb{Z}}v_{1,n-m}^1\omega_{m}^1.
\]
First, it follows from Proposition \ref{0mode}, the Sobolev's embedding inequalities and the estimate \eqref{eq722} that
\begin{equation}\begin{aligned}\label{eq724}
&\|\tilde{\Bv}_0\|_{H^2(\Omega)}\le C_2\left\|F^1_{1,0}-F^2_{1,0}\right\|_{L^2(\Omega)}
\le\left\|\sum\limits_{m\neq 0}v_{2,-m}^1\omega_{m}^1-v_{2,-m}^2\omega_{m}^2\right\|_{L^2(\Omega)}\\
\le&C\left\|\sum\limits_{m\neq 0}v_{2,-m}^1\left(\omega_{m}^1-\omega_{m}^2\right)\right\|_{L^2(\Omega)}+C\left\|\sum\limits_{m\neq 0}\left(v_{2,-m}^1-v_{2,-m}^2\right)\omega_{m}^2\right\|_{L^2(\Omega)}\\
\le& C\left\|\left(\Q v_{2}^1\cdot \Q(\omega^1-\omega^2)\right)_0\right\|_{L^2(\Omega)}+C\left\|\left(\Q(v_{2}^1-v_{2}^2)\cdot\Q\omega^2\right)_0\right\|_{L^2(\Omega)}\\
\le&C\left\|\Q\Bv^{1}\right\|_{L^6(\Omega)}\left\|\Q\tilde{\omega}\right\|_{L^3(\Omega)}
+C\left\|\Q\tilde{\Bv}\right\|_{L^\infty(\Omega)}\left\|\Q\Bv^{2}\right\|_{H^1(\Omega)}\\
\le&C\left\|\Q\Bv^{1}\right\|_{L^2(\Omega)}^\frac25\left\|\Q\Bv^{1}\right\|_{H^\frac53(\Omega)}^\frac35
\left\|\Q\tilde{\Bv}\right\|_{H^\frac32(\Omega)}
\\
&+C\left\|\Q\tilde{\Bv}\right\|_{H^\frac32(\Omega)}\left\|\Q\Bv^{2}\right\|_{L^2(\Omega)}^\frac25\left\|\Q\Bv^{2}\right\|_{H^\frac53(\Omega)}^\frac35\\ \le&C\Phi^{-\frac{103}{480}}\left\|\Q\tilde{\Bv}\right\|_{H^\frac32(\Omega)}.
\end{aligned}\end{equation}
In addition, applying Propositions \ref{mediumv} and \ref{highv}, and the interpolation inequalities yields
\begin{equation}\begin{aligned}\label{eq725}
&\|\Q\tilde{\Bv}\|_{H^\frac32(\Omega)}\le C\|\Q\tilde{\Bv}\|_{L^2(\Omega)}^\frac{1}{10}\|\Q\tilde{\Bv}\|_{H^\frac53(\Omega)}^\frac{9}{10}\\
\le& C\Phi^{-\frac{1}{15}}\left(\sum\limits_{n\neq 0}\left\|F^{1}_{2,n}-F^{2}_{2,n}\right\|_{L^2(\Omega)}^2
+\left\|F^{1}_{1,n}-F^{2}_{1,n}\right\|_{L^2(\Omega)}^2\right)^\frac12.
\end{aligned}\end{equation}
It follows from the Hausdorff-Young inequality and Sobolev's embedding inequalities that one has
\begin{equation}\begin{aligned}\label{eq726}
&\left(\sum\limits_{n\neq0}\left\| \sum\limits_{m\neq 0,n}v^1_{2,n-m}\omega^1_{m}-v^2_{2,n-m}\omega^2_{m}\right\|_{L^2(\Omega)}^2\right)^\frac12\\
\le&C\left(\sum\limits_{n\neq 0}\left\|\sum\limits_{m\neq 0,n}v^1_{2,n-m}(\omega^1_m-\omega^2_m)\right\|_{L^2(\Omega)}^2
+\sum\limits_{n\neq 0}\left\|\sum\limits_{m\neq 0,n}(v^1_{2,n-m}-v^2_{2,n-m})\omega^2_m\right\|_{L^2(\Omega)}^2\right)^\frac12\\
\le&C\left\|\Q \Bv^1\cdot\Q\tilde{\omega}\right\|_{L^2(\Omega)}+C\left\|\Q \tilde{\Bv}\cdot\Q\omega^2\right\|_{L^2(\Omega)}\\
\le&C\left\|\Q \Bv^1\right\|_{L^6(\Omega)}\left\|\Q\tilde{\omega}\right\|_{L^3(\Omega)}+C\left\|\Q \tilde{\Bv}\right\|_{L^6(\Omega)}\left\|\Q\omega^2\right\|_{L^3(\Omega)}\\
\le & C\left\|\Q v^1\right\|_{H^\frac53(\Omega)}\left\|\Q\tilde{\Bv}\right\|_{H^\frac32(\Omega)}+C\left\|\Q v^2\right\|_{H^\frac53(\Omega)}\left\|\Q\tilde{\Bv}\right\|_{H^\frac32(\Omega)}\leq C\Phi^\frac{1}{32}\left\|\Q\tilde{\Bv}\right\|_{H^\frac32(\Omega)}
\end{aligned}\end{equation}
and
\begin{equation}\begin{aligned}\label{eq727}
&\left(\sum\limits_{n\neq 0}\left\|\sum\limits_{m\neq 0,n}v^1_{1,n-m}\omega^1_{m}-v^2_{1,n-m}\omega^2_{m}\right\|_{L^2(\Omega)}^2\right)^\frac12\\
\le&C\left\|\Q \Bv^1\cdot\Q\tilde{\omega}\right\|_{L^2(\Omega)}+C\left\|\Q \tilde{\Bv}\cdot\Q\omega^2\right\|_{L^2(\Omega)}\le C\Phi^\frac{1}{32}\left\|\Q\tilde{\Bv}\right\|_{H^\frac32(\Omega)}.
\end{aligned}\end{equation}
Similarly, one has
\begin{equation}\begin{aligned}\label{eq728}
&\left(\sum\limits_{n\neq 0}\left\|v^1_{2,n}\omega^1_{0}-v^2_{2,n}\omega^2_{0}\right\|_{L^2(\Omega)}^2\right)^\frac12\\
\le&C\left(\sum\limits_{n\neq 0}\left\|v^1_{2,n}\tilde{\omega}_{0}\right\|_{L^2(\Omega)}^2\right)^\frac12+
C\left(\sum\limits_{n\neq 0}\left\|\tilde{v}_{2,n}\omega^2_{0}\right\|_{L^2(\Omega)}^2\right)^\frac12\\
\le&C\left\|\tilde{\omega}_{0}\right\|_{L^3(\Omega)}\left(\sum\limits_{n\neq 0}\left\|v^1_{2,n}\right\|_{L^6(\Omega)}^2\right)^\frac12+
C\left\|\omega^2_{0}\right\|_{L^3(\Omega)}\left(\sum\limits_{n\neq 0}\left\|\tilde{v}_{2,n}\right\|_{L^6(\Omega)}^2\right)^\frac12\\
\le&C\left\|\tilde{\Bv}_{0}\right\|_{H^2(\Omega)}\left(\sum\limits_{n\neq 0}\left\|\Bv^{1}_n\right\|_{H^\frac32(\Omega)}^2\right)^\frac12+
C\left\|\Bv^{2}_{0}\right\|_{H^2(\Omega)}\left(\sum\limits_{n\neq 0}\left\|\tilde{\Bv}_n\right\|_{H^\frac32(\Omega)}^2\right)^\frac12\\
\le&C\Phi^\frac{1}{32}\left\|\tilde{\Bv}_0\right\|_{H^2(\Omega)}+C\Phi^\frac{1}{32}\left\|\Q\tilde{\Bv}\right\|_{H^\frac32(\Omega)}.
\end{aligned}\end{equation}
Finally, it holds that
\begin{equation}\begin{aligned}\label{eq729}
\left(\sum\limits_{n\neq 0}\left\|v^1_{1,n}\omega^1_{0}-v^2_{1,n}\omega^2_{0}\right\|_{L^2(\Omega)}^2\right)^\frac12\le&C\Phi^\frac{1}{32}\left\|\tilde{\Bv}_0\right\|_{H^2(\Omega)}+C\Phi^\frac{1}{32}\left\|\Q\tilde{\Bv}\right\|_{H^\frac32(\Omega)}
\end{aligned}\end{equation}
and
\begin{equation}\begin{aligned}\label{eq730}
&\left(\sum\limits_{n\neq 0}\left\|v^1_{1,0}\omega^1_{n}-v^2_{1,0}\omega^2_{n}\right\|_{L^2(\Omega)}^2\right)^\frac12\\
\le&C\left\|v_{1,0}^1\right\|_{L^6(\Omega)}\left(\sum\limits_{n\neq 0}\left\|\tilde{\omega}_n\right\|_{L^3(\Omega)}^2\right)^\frac12+
C\left\|\tilde{v}_{1,0}\right\|_{L^6(\Omega)}\left(\sum\limits_{n\neq 0}\left\|\omega^2_n\right\|_{L^3(\Omega)}^2\right)^\frac12\\
\le&C\Phi^\frac{1}{32}\left\|\Q\tilde{\Bv}\right\|_{H^\frac32(\Omega)}+C\Phi^\frac{1}{32}\left\|\tilde{\Bv}_0\right\|_{H^2(\Omega)}.
\end{aligned}\end{equation}

Combining the estimates \eqref{eq724}-\eqref{eq730} gives
\begin{equation}\nonumber
\begin{aligned}
\|\tilde{\Bv}_0\|_{H^2(\Omega)}+\|\Q\tilde{\Bv}\|_{H^\frac32(\Omega)}\le& C\Phi^{-\frac{17}{480}}\left(\|\tilde{\Bv}_0\|_{H^2(\Omega)}+\|\Q\tilde{\Bv}\|_{H^\frac32(\Omega)}\right).
\end{aligned}\end{equation}
This implies that
\begin{equation}\nonumber
\begin{aligned}
\|\tilde{\Bv}_0\|_{H^2(\Omega)}+\|\Q\tilde{\Bv}\|_{H^\frac32(\Omega)}=0,
\end{aligned}\end{equation}
provided that $\Phi\ge \Phi_0$ is sufficiently large. This finishes the proof of Proposition \ref{largeF}.
\end{proof}
Taking $\Bu=\Bv+\BU$, the proof of Theorem \ref{largeforce} is completed.

%%%%%%%%%%%%%%%%%%%%%% Uniqueness of Poiseuille flow %%%%%%%%%%%%%%%%%%%%%%%%%%%%%%%%%%%%

\section{Uniqueness of the solutions for the nonlinear problem}\label{sec-unique}
This section devotes to the proof of Theorem \ref{uniqueness}.

Let $\Bu$ be a periodic solution of the Navier-Stokes equation \eqref{NS}-\eqref{BC} and $\Bv=\Bu-\BU$. Hence $\Bv$ satisfies the perturbed problem \eqref{model11}-\eqref{model11'}. In the case (a), if $\Bu$ satisfies
\eqref{est6-0}, then it holds that
\begin{equation}\label{6-2}
\|v_2\|_{H^1(\Omega)}\le \varepsilon_1.
\end{equation}
According to the linear estimates obtained in Propositions \ref{0mode}, \ref{mediumv}, and \ref{highv}, one has
\begin{equation}\label{6-3}
\left\|\Bv_0\right\|_{H^2(\Omega)}+\left\|\Q\Bv\right\|_{H^\frac53(\Omega)}\leq C\left\|(\Bv\cdot\nabla v_1)_0\right\|_{L^2(\Omega)}+C\left\|\Q(\Bv\cdot\nabla \Bv)\right\|_{L^2(\Omega)},
\end{equation}
where $\Q$ is the projection operator defined in \eqref{projection}.
It follows from the divergence free condition $\mathrm{div}\,\Bv=\partial_xv_1+\partial_yv_2=0$ that one can rewrite the nonlinear term $\Bv\cdot\nabla\Bv$ as
\begin{equation}\nonumber
\Bv\cdot\nabla\Bv=(-v_1\partial_yv_2+v_2\partial_yv_1,v_1\partial_xv_2+v_2\partial_yv_2)^T.
\end{equation}
Since $v_{2,0}=0$, using Hausdorff-Young inequality yields
\begin{equation}\label{6-4}
\begin{aligned}
\|(\Bv\cdot\nabla v_1)_0\|_{L^2} \leq &~~ \|(v_1\partial_yv_2)_0\|_{L^2(\Omega)}+\|(v_2\partial_yv_1)_0\|_{L^2(\Omega)}\\
\leq &~~ \left\|\sum_{n \neq0}v_{1,-n}v_{2,n}'\right\|_{L^2(\Omega)}+\left\|\sum_{n \neq0}v_{2,-n}v_{1,n}'\right\|_{L^2(\Omega)}\\
\leq &~~ \|\Q v_1\|_{L^\infty(\Omega)}\|\Q(\partial_yv_2)\|_{L^2(\Omega)}+\|\Q v_2\|_{L^6(\Omega)}\|\Q(\partial_yv_1)\|_{L^3(\Omega)}\\
\leq &~~ C\|v_2\|_{H^1(\Omega)}\|\Q \Bv\|_{H^\frac32(\Omega)}.
\end{aligned}\end{equation}
Moreover, one has
\begin{equation}\label{6-5}\begin{aligned}
\|\Q(\Bv\cdot\nabla\Bv)\|_{L^2(\Omega)}\leq &~~ \|\Q(v_1\partial_yv_2)\|_{L^2(\Omega)}+\|\Q(v_2\partial_yv_1)\|_{L^2(\Omega)}
+\|\Q(v_1\partial_xv_2)\|_{L^2(\Omega)}\\
&~~+\|\Q(v_2\partial_yv_2)\|_{L^2(\Omega)}.
\end{aligned}\end{equation}
Using Hausdorff-Young inequality again yields
\begin{equation}\nonumber
\begin{aligned}
\|\Q(v_1\partial_yv_2)\|_{L^2(\Omega)}\leq&~~\left(\sum_{n\neq0}\left\|\sum_{m\in \Z}v_{1,n-m}v_{2,m}'\right\|_{L^2(\Omega)}^2\right)^\frac12\\
\leq&~~\left(\sum_{n\neq0}\left\|\sum_{m\neq 0,n}v_{1,n-m}v_{2,m}'\right\|_{L^2(\Omega)}^2\right)^\frac12
+\left(\sum_{n\neq0}\left\|v_{1,0}v_{2,n}'\right\|_{L^2(\Omega)}^2\right)^\frac12\\
\leq&~~\|\Q v_{1}\|_{L^\infty(\Omega)}\|\Q(\partial_yv_{2})\|_{L^2(\Omega)}
+\|\Bv_0\|_{L^\infty(\Omega)}\|\Q(\partial_yv_{2})\|_{L^2(\Omega)}\\
\leq&~~C\|v_2\|_{H^1(\Omega)}(\|\Bv_0\|_{H^2(\Omega)}+\|\Q\Bv\|_{H^\frac32(\Omega)}).
\end{aligned}\end{equation}
Since the other terms on the right hand side of \eqref{6-5} can be estimated similarly, one finally obtains
\begin{equation}\label{6-7}
\|\Q(\Bv\cdot\nabla\Bv)\|_{L^2(\Omega)}\leq C\|\Q v_2\|_{H^1(\Omega)}(\|\Bv_0\|_{H^2(\Omega)}+\|\Q\Bv\|_{H^\frac32(\Omega)}).
\end{equation}
Combining \eqref{6-3}-\eqref{6-4} and \eqref{6-7} gives
\begin{equation}\nonumber
\left\|\Bv_0\right\|_{H^2(\Omega)}+\left\|\Q\Bv\right\|_{H^\frac53(\Omega)}\\
\leq C\varepsilon_1(\left\|\Bv_0\right\|_{H^2(\Omega)}+\left\|\Q\Bv\right\|_{H^\frac53(\Omega)}).
\end{equation}
Hence
\begin{equation}\nonumber
\left\|\Bv_0\right\|_{H^2(\Omega)}+\left\|\Q\Bv\right\|_{H^\frac53(\Omega)}=0,
\end{equation}
provided that $\varepsilon_1$ is sufficiently small.

In the case (b), it also follows from the interpolation inequality and the linear estimates obtained in Propositions \ref{0mode}, \ref{mediumv}, and \ref{highv}  that one has
\begin{equation}\begin{aligned}\label{6-10}
&~~\left\|\Bv_0\right\|_{H^2(\Omega)}+\Phi^\frac{1}{30}\left\|\Q\Bv\right\|_{H^\frac32(\Omega)}\\
\leq&~~\left\|\Bv_0\right\|_{H^2(\Omega)}+C\Phi^\frac{1}{30}\left\|\Q\Bv\right\|_{L^2(\Omega)}^\frac{1}{10}\left\|\Q\Bv\right\|_{H^\frac53(\Omega)}^\frac{9}{10}\\
\leq&~~C\left\|(\Bv\cdot\nabla v_1)_0\right\|_{L^2(\Omega)}+C\Phi^{-\frac{1}{30}}\left\|\Q(\Bv\cdot\nabla \Bv)\right\|_{L^2(\Omega)}.
\end{aligned}\end{equation}
The estimate \eqref{6-10}, together with  \eqref{6-4},\eqref{6-7} and the assumption $\|v_2\|_{H^1(\Omega)}=\|u_2\|_{H^1(\Omega)}\leq \Phi^\frac{1}{60}$, yields
\begin{equation}\nonumber
\begin{aligned}
&~~\left\|\Bv_0\right\|_{H^2(\Omega)}+\Phi^\frac{1}{30}\left\|\Q\Bv\right\|_{H^\frac32(\Omega)}\\
\leq&~~C\Phi^\frac{1}{60}\left\|\Q\Bv\right\|_{H^\frac32(\Omega)}+C\Phi^{-\frac{1}{60}}(\left\|\Bv_0\right\|_{H^2(\Omega)}+\left\|\Q\Bv\right\|_{H^\frac32(\Omega)}).
\end{aligned}\end{equation}
This implies that
\begin{equation}\nonumber
\left\|\Bv_0\right\|_{H^2(\Omega)}+\Phi^\frac{1}{30}\left\|\Q\Bv\right\|_{H^\frac32(\Omega)}=0,
\end{equation}
provided $\Phi$ is sufficiently large. Hence the proof of Theorem \ref{uniqueness} is completed.

%%%%%%%%%%%%%%%%%%%%%%%%%%Appendix%%%%%%%%%%%%%%%%%%%%%%%%%%%%%%%%%%%%%%%%%%%%%%

%\newpage

\appendix
\section{Some elementary lemmas}
In this appendix, we collect some elementary lemmas which play important roles in the paper and might be useful elsewhere.
We first give some Poincar\'e type inequalities.
\begin{lemma}\label{lemmaA1}
For a  function $g\in C^2([-1,1])$ satisfying $g(-1)=g(1)=0$,
it holds that
\be \label{A1-1}
\int_{-1}^1 |g |^2  \, dy \leq  \int_{-1}^1 |g'|^2 \, dy.
\ee
Moreover, one has
\be\label{A1-2}
\inte |g'|^2 \, dy \leq \left( \inte | g''|^2\, dy \right)^{\frac12} \left( \inte |g|^2 \, dy \right)^{\frac12}
\ee
and consequently,
\be\label{A1-3}
\inte |g'|^2 \, dy
\leq \inte | g''|^2  \, dy.
\ee
\end{lemma}
\begin{proof}
The inequality \eqref{A1-1} is owing to Poincar\'e's inequality. Integration by parts and using the homogeneous boundary conditions for $g$ give
\be \nonumber
\inte | g'|^2  \, dy= -\inte g''\overline{g } \, dy
\leq \left( \inte | g''|^2\, dy \right)^{\frac12}
\left( \inte |g |^2\, dy   \right)^{\frac12}.
\ee
The estimate \eqref{A1-3} is the direct consequence of \eqref{A1-1} and \eqref{A1-2}. Hence, the proof of the lemma is completed.
\end{proof}

%%%%%%%%%%%%%%%%%%%%%%%%%%%%%%%%%boundary value%%%%%%%%%%%%%%%%%%%%%%%%%%%%%%%%%%%%%
In the following lemma, we give the pointwise estimate for the functions evaluated on the boundary.
\begin{lemma}\label{lemmaA2}
For a function $g\in C^4([-1,1])$ with $g(\pm1)=0$, one has
\begin{equation}\nonumber
|g' (\pm1) |  \leq C\left(   \inte| g' |^2 \, dy \right)^{\frac14}  \left( \inte \left|g'' \right|^2  \, dy       \right)^{\frac14}
\end{equation}
and
\begin{equation}\label{A2-2}
\inte| g^{(3)} |^2 \, dy \leq C\left(   \inte| g^{(4)}|^2 \, dy \right)^{\frac12}  \left( \inte \left|g'' \right|^2  \, dy       \right)^{\frac12}+C\inte \left|g'' \right|^2  \, dy  .
\end{equation}
\end{lemma}

\begin{proof}
For every $y\in [-1, 1] $, one has
\be \label{A2-3}\ba
|g'(1)|^2=& | g'(y)|^2  + \int_y^1 g''\overline{g'} \, ds
+ \int_y^1 g'\overline{g''} \, ds\\
\le&| g'(y)|^2+2\left(\inte|g''|^2\,dy\right)^\frac12\left(\inte|g'|^2\,dy\right)^\frac12.
\ea
\ee
Integrating \eqref{A2-3} over $[-1,1]$ yields
\be \label{A2-4} \begin{aligned}
2|g'(1)|^2\leq& \inte| g'(y)|^2\,dy+4\left(\inte|g''|^2\,dy\right)^\frac12\left(\inte|g'|^2\,dy\right)^\frac12\\
\leq& 5\left(\inte|g''|^2\,dy\right)^\frac12\left(\inte|g'|^2\,dy\right)^\frac12.
\end{aligned}\ee
Similarly,
\be \nonumber
2|g'(-1)|^2\le 5\left(\inte|g''|^2\,dy\right)^\frac12\left(\inte|g'|^2\,dy\right)^\frac12.
\ee
Similar to \eqref{A2-4}, one has
\be \label{A2-6}\ba
2|g''(\pm1)|^2
\le&\inte | g''(y)|^2\,dy+4\left(\inte|g''|^2\,dy\right)^\frac12\left(\inte|g^{(3)}|^2\,dy\right)^\frac12
\ea
\ee
and
\be \label{A2-7}\ba
2|g^{(3)}(\pm1)|^2
\le&\inte | g^{(3)}(y)|^2\,dy+4\left(\inte|g^{(3)}|^2\,dy\right)^\frac12\left(\inte|g^{(4)}|^2\,dy\right)^\frac12.
\ea
\ee
Applying integration by parts, \eqref{A2-6} and \eqref{A2-7} yields
\begin{equation}\nonumber
\begin{aligned}
&\inte |g^{(3)}|^2\,dy=-\inte g''\overline{g^{(4)}}\,dy+g''(1)\overline{g^{(3)}(1)}-g''(-1)\overline{g^{(3)}(-1)}\\
\leq & \left(\inte|g''|^2\,dy\right)^\frac12\left(\inte|g^{(4)}|^2\,dy\right)^\frac12+C\left(\inte|g^{(3)}|^2\,dy\right)^\frac12\left(\inte|g''|^2\,dy\right)^\frac12\\
&+C\left(\inte|g^{(3)}|^2\,dy\right)^\frac34\left(\inte|g''|^2\,dy\right)^\frac14\\
&+C\left(\inte|g''|^2\,dy\right)^\frac12\left(\inte|g^{(3)}|^2\,dy\right)^\frac14\left(\inte|g^{(4)}|^2\,dy\right)^\frac14\\
&+C\left(\inte|g''|^2\,dy\right)^\frac14\left(\inte|g^{(3)}|^2\,dy\right)^\frac12\left(\inte|g^{(4)}|^2\,dy\right)^\frac14.
\end{aligned}
\end{equation}
Using Young's inequality, one proves \eqref{A2-2}. Hence, we finish the proof of the lemma.
\end{proof}
%%%%%%%%%%%%%%%%%%%%%%%%%%%%Hardy%%%%%%%%%%%%%%%%%%%%%%%%%%

The following lemma is a variant of Hardy-Littlewood-P\'olya type inequality \cite[pp.165]{HLP}.
\begin{lemma}\label{lemmaHLP}
Let $g\in C^1([-1,1])$, one has
\begin{equation}\nonumber
\inte |g(y)|^2\,dy\le \frac13\inte |g'(y)|^2(1-y^2)\,dy +92\inte |g(y)|^2(1-y^2)\,dy.
\end{equation}

\end{lemma}

\begin{proof}
Let
\be \nonumber
h(t)=g\left(\frac{1+t}{2}\right) \text{ for } t\in [0,1].
\ee
It follows from \cite[Lemma A.3]{WX1} that one has

\be \nonumber
\int_0^1 |h(t)-h(0)|^2\,dt\leq \frac12 \int_0^1 |h'(t)|^2(1-t^2)\,dt.
\ee
This implies
\be\nonumber
\int_\frac12^1 \left|g(y)-g\left(\frac12\right)\right|^2\,dy\leq \frac12 \int_\frac12^1|g'(y)|^2(y-y^2)\,dy
\leq \frac14\int_\frac12^1|g'(y)|^2(1-y^2)\,dy.
\ee
Hence,
\be\label{A3-6}\begin{aligned}
\int_\frac12^1 |g(y)|^2\,dy\leq& \frac43\int_\frac12^1 \left|g(y)-g\left(\frac12\right)\right|^2\,dy+4\int_\frac12^1 \left|g\left(\frac12\right)\right|^2\,dy\\
\leq &\frac13\int_\frac12^1|g'(y)|^2(1-y^2)\,dy+2\left|g\left(\frac12\right)\right|^2.
\end{aligned}\ee
For any $y\in [0,\frac12]$, one has
\be\nonumber
\left|g\left(\frac12\right)\right|^2=\int_y^\frac12g'(s)\overline {g(s)}+g(s)\overline{g'(s)}\,ds+|g(y)|^2.
\ee
Integrating equation above over $[0,\frac12]$, one has
\be\label{A3-7}
\ba
\left|g\left(\frac12\right)\right|^2\leq& 4\int_0^\frac12 |g'(s)||g(s)|\,ds+2\int_0^\frac12|g(y)|^2\,dy\\
\leq&\frac18\int_0^\frac12|g'(y)|^2\,dy +34\int_0^\frac12|g(y)|^2\,dy\\
\leq&\frac16\int_0^\frac12 |g'(y)|^2(1-y^2)\,dy +\frac{136}{3}\int_0^\frac12|g(y)|^2(1-y^2)\,dy.
\ea\ee

One the other hand, it's easy to obtain
\be\label{A3-8}
\int_0^\frac12 |g(y)|^2\,dy\le\frac43\int_0^\frac12|g(y)|^2(1-y^2)\,dy.
\ee
Combining the estimates \eqref{A3-6}-\eqref{A3-8} gives
\be\nonumber
\int_0^1  |g(y)|^2\,dy\le \frac13\int_0^1 |g'(y)|^2(1-y^2)\,dy +92\int_0^1|g(y)|^2(1-y^2)\,dy.
\ee
The estimate over $[-1,0]$ is similar. Hence, the proof of this lemma is completed.

\end{proof}

%%%%%%%%%%%%%%%%%Lemma A6%%%%%%%%%%%%%%
The following lemma is about a weighted interpolation inequality, which are quite similar to \cite[inequality (3.28)]{GHM}.
\begin{lemma}\label{weightinequality} Let $g \in C^1([-1, 1])$, then one has
\begin{equation}\nonumber
\int_{-1}^1 |g|^2 \, dy  \leq C \left(\int_{-1}^1 (1-y^2)|g|^2  \, dy\right)^{\frac23} \left(\int_{-1}^1 |g'|^2 \, dy\right)^{\frac13}+C\int_{-1}^1 (1-y^2)|g|^2  \, dy.
\end{equation}
\end{lemma}

\begin{proof}
Let $\delta \in(0,1)$ be a parameter to be determined. Then one has
\be\nonumber
\ba
\inte|g|^2\,dy\leq&~~2\int_{|y|\leq \frac12}(1-y^2)|g|^2\,dy+\int_{\frac12\le |y|\le 1-\delta}\frac{1-y^2}{\delta}|g|^2\,dy+\int_{1-\delta\le|y|\le 1}|g|^2\,dy\\
\leq&~~2\inte(1-y^2)|g|^2\,dy+\frac{2}{\delta}\inte(1-y^2)|g|^2\,dy+2\delta\sup_{y\in[-1,1]}|g(y)|^2.
\ea\ee
Taking
\[\delta=\frac{\left(\inte(1-y^2)|g|^2\,dy\right)^\frac12}{\sup_{y\in[-1,1]}|g(y)|+\left(\inte(1-y^2)|g|^2\,dy\right)^\frac12}\]
yields
\be\label{A4-3}
\inte|g|^2\,dy\leq 4\inte(1-y^2)|g|^2\,dy+4\left(\inte(1-y^2)|g|^2\,dy\right)^\frac12\sup_{y\in[-1,1]}|g(y)|.
\ee
Note that for any $y,t\in[-1,1]$, one has
\be\label{A4-4}
|g(y)|^2=|g(t)|^2+\int_{t}^y\frac{d}{ds}|g(s)|^2\,ds\leq |g(t)|^2+ 2\left(\inte|g|^2\,ds\right)^\frac12\left(\inte|g'|^2\,ds\right)^\frac12.
\ee
Integrating the  equation \eqref{A4-4} with respect to $t$ over $[-1,1]$ yields
\be\label{A4-5}
|g(y)|^2\leq \frac12\inte|g|^2\,dy+2\left(\inte|g|^2\,dy\right)^\frac12\left(\inte|g'|^2\,dy\right)^\frac12.
\ee
Then it follows from \eqref{A4-3} and \eqref{A4-5} that
\be\nonumber
\ba
\inte|g|^2\,dy\leq&~~C\inte(1-y^2)|g|^2\,dy+C\left(\inte(1-y^2)|g|^2\,dy\right)^\frac12\left(\inte|g|^2\,dy\right)^\frac12\\
&~~+C\left(\inte(1-y^2)|g|^2\,dy\right)^\frac12\left(\inte|g|^2\,dy\right)^\frac14\left(\inte|g'|^2\,dy\right)^\frac14.
\ea\ee
An application of Young's inequality implies that
\be\nonumber
\inte|g|^2\,dy\leq C\left(\inte(1-y^2)|g|^2\,dy\right)^\frac23\left(\inte|g'|^2\,ds\right)^\frac13+C\inte(1-y^2)|g|^2\,dy.
\ee
This finishes the proof of the lemma.
\end{proof}

The following lemma is inspired by \cite[Lemma 9.3 ]{CWZ}. Here we give a different elementary proof.
\begin{lemma}\label{interpolation}
Assume that $g\in C^1([-1,1])$ and $g(\pm1)=0$. It holds that
\begin{equation}\label{A5-1}
\inte |g|^2\,dy \leq C\left(\inte |(1-y^2) g|^2\,dy\right)^\frac12\left(\inte |g'|^2\,dy\right)^\frac12.
\end{equation}
\end{lemma}
\begin{proof}
If
\begin{equation}\nonumber
\inte |g'|^2\,dy\leq \inte |(1-y^2) g|^2\,dy,
\end{equation}
 the inequality \eqref{A5-1} follows from Lemma \ref{lemmaA1}.

Otherwise, let $\delta\in(0,1)$ be a constant to be determined. Thus one has
\begin{equation}\label{A5-2}
\int_0^1 |g|^2\,dy\leq\frac{1}{\delta^2}\int_{0}^{1-\delta}|(1-y^2)g|^2\,dy+\delta\sup_{1-\delta<y\leq1}|g(y)|^2.
\end{equation}
Since $g(1)=0$, it follows that
\begin{equation}\label{A5-3}
\sup_{1-\delta<y\leq1}|g(y)|\leq \int_{1-\delta}^1|g'(y)|\,dy\leq \delta^\frac12\left(\int_{1-\delta}^1|g'(y)|^2\,dy\right)^\frac12.
\end{equation}
Combining \eqref{A5-2} and \eqref{A5-3} gives
\begin{equation}\label{A5-4}
\int_0^1 |g|^2\,dy\leq\frac{1}{\delta^2}\int_{0}^1|(1-y^2)g|^2\,dy+\delta^2\int_0^1|g'(y)|^2\,dy.
\end{equation}
Similarly, one can also obtain
\begin{equation}\label{A5-5}
\int_{-1}^0 |g|^2\,dy\leq\frac{1}{\delta^2}\int_{-1}^0|(1-y^2)g|^2\,dy+\delta^2\int_{-1}^0|g'(y)|^2\,dy.
\end{equation}
Summing \eqref{A5-4} and \eqref{A5-5} and taking
\begin{equation}\nonumber
\delta=\left(\frac{\inte|(1-y^2)g|^2\,dy}{\inte|g'(y)|^2\,dy}\right)^\frac14\in (0,1),
\end{equation}
 one has
\begin{equation}\nonumber
\begin{aligned}
\inte |g|^2\,dy\leq&\frac{1}{\delta^2}\inte|(1-y^2)g|^2\,dy+\delta^2\inte|g'(y)|^2\,dy \\
\leq& C\left(\inte |(1-y^2) g|^2\,dy\right)^\frac12\left(\inte |g'|^2\,dy\right)^\frac12.
\end{aligned}\end{equation}
This finishes the proof of Lemma \ref{interpolation}.
\end{proof}

The following lemma on the Airy function gives the estimate for the boundary layers constructed in Section \ref{secinter}.
\begin{lemma}\label{airy-est}
(1) Let $C_{0, n,\Phi}$  be given in \eqref{5-3-62} and $G_{n,\Phi}$ be defined in \eqref{5-3-61}.  It holds that
\begin{equation}\nonumber
\tilde{C}_{0}:=\inf \left\{\left|C_{0, n, \Phi}\right|:\Phi\ge  1,1\le |n| \leq L_0 \Phi^{\frac{1}{2}}\right\}>0
\end{equation}
and
\begin{equation}\nonumber
\sup _{\Phi \geq 1} \sup _{\frac1L \leq|\hat{n}| \leq \Phi^{\frac12}}  \sup _{\rho \ge R} e^{\rho}\left|\frac{d^k G_{n, \Phi}}{d \rho^k}(\rho)\right|<\infty, ~~k=0,1,2,3,
\end{equation}
provided that $R$ is sufficiently large.

(2) There exists a constant $\epsilon\in (0,1)$ such that defining
\begin{equation}\nonumber
\Sigma_{\epsilon}:=\left\{\mu \in \mathbb{C} | \text { arg } \mu=-\frac{\pi}{6}, 0 \leq|\mu| \leq \epsilon\right\},
\end{equation}
then
\begin{equation}\nonumber
K_{\epsilon}:=\inf _{\mu \in \Sigma_{\epsilon}}\left|\int_\ell e^{-\mu z} Ai\left(z+\mu^{2}\right) \,dz\right|\ge\frac16,
\end{equation}
where $\ell$ is the contour $\ell=\left\{re^{i\frac{\pi}{6}}|r\ge 0\right\}$.
\end{lemma}
Since the proof of Lemma \ref{airy-est} is exactly the same as that of \cite[Lemma 3.7]{GHM}, we omit the proof here.

The following elementary fixed point lemma is the basic tool to prove the nonlinear structural stability. The proof is given in \cite{WX1}.
\bl \label{lemmaA6} Let $Y$ be a Banach space with the norm $\|\cdot\|_{Y}$, and $\mathcal{B}:\  Y\times Y\rightarrow Y$ be a bilinear map. If for all $\Bz_1, \Bz_2 \in Y$, one has
\be \nonumber
\|\mathcal{B}(\Bz_1, \Bz_2) \|_{Y} \leq \eta \|\Bz_1\|_{Y}  \|\Bz_2\|_{Y},
\ee
then for all $\Bz^* \in Y$ satisfying $4 \eta \| \Bz^* \|_{Y} < 1$, the equation
\be \nonumber
\Bz = \Bz^*  + \mathcal{B}(\Bz, \Bz)
\ee
has a unique solution $\Bz \in Y$ satisfying
\be \nonumber
\|\Bz\|_{Y} \leq 2\|\Bz^*\|_{Y}.
\ee
\el

%%%%%%%%%%%%%%%%%%%%%%%%%%%%%%%%%%%%%%%%%%%%%%%%%%%%%%%%%%%%%%%%%%%%

{\bf Acknowledgement.}
The research of Wang was partially supported by NSFC grant 11671289. The research of  Xie was partially supported by  NSFC grants 11971307 and 11631008,  and Young Changjiang Scholar of Ministry of Education in China.

\end{document}